\documentclass[10pt]{article}

\usepackage{amsmath}
\usepackage{amsfonts}
\usepackage{amsthm}
\usepackage{subcaption}

\usepackage[sort]{cite} 

\usepackage[margin=1in]{geometry}

\usepackage{graphicx}
\usepackage{tikz}
\usepackage{tabularx}

\usepackage{xcolor}

\newcommand{\beql}[1]{\begin{equation}\label{#1}}
\newcommand{\eeql}{\end{equation}}
\newcommand{\eqn}[1]{(\ref{#1})}

\usepackage{amsthm}
\usepackage{amsmath}
\usepackage{amsfonts}
\usepackage{amssymb}
\usepackage{dsfont,color,xfrac,enumitem,mathrsfs}
\usepackage{mathtools,tikz-cd}

\usepackage[numbers]{natbib}
\usepackage[colorlinks,citecolor=blue,urlcolor=blue]{hyperref}
\usepackage{graphicx}

\numberwithin{equation}{section}

\newtheorem{theorem}{Theorem}[section]
\newtheorem{lemma}[theorem]{Lemma}
\newtheorem{thm}{Theorem}[section]

\newtheorem{proposition}[theorem]{Proposition}

\newtheorem{conj}[theorem]{Conjecture}   
\theoremstyle{remark}
\newtheorem{definition}[theorem]{Definition}

\newtheorem{rem}[theorem]{Remark}
\newtheorem{obs}[thm]{Observation}

\renewcommand{\le}{\leqslant} 
\renewcommand{\ge}{\geqslant} 
\renewcommand{\leq}{\leqslant} 
\renewcommand{\geq}{\geqslant}

\DeclareMathOperator{\E}{\mathds{E}}
\DeclareMathOperator{\pr}{\mathds{P}}

\begin{document}

\title{Data Flow Dissemination in a Network}



\author
{
Aditya Gopalan \\
University of Illinois \\ 
at Urbana-Champaign\\
\texttt{gopalan6@illinois.edu}
\and
Alexander L. Stolyar \\
University of Illinois \\ 
at Urbana-Champaign\\
\texttt{stolyar@illinois.edu}
}

\date{\today}

	\maketitle
	
	\abstract{
	We consider the following network model motivated, in particular, by blockchains and peer-to-peer live streaming. 
	Data packet flows arrive at the network nodes and need to be disseminated to all other nodes. 
	Packets are relayed through the network via links of finite capacity. A packet leaves the network when it is disseminated to all nodes. 
	Our focus is on two communication disciplines, which determine the order in which packets are transmitted over each link, namely {\em Random-Useful} (RU) and {\em Oldest-Useful} (OU).
	We show that RU has the maximum stability region 
	in a general network. For the OU we demonstrate that, somewhat surprisingly, it does {\em not} in general have the maximum stability region. 
	We prove that OU does achieve maximum stability in the important special case of a symmetric network, given by the full graph with equal capacities on all links and equal arrival rates at all nodes. We also give other stability results, and compare different disciplines' performances in a symmetric system via simulation. 
	
	Finally, we study the cumulative delays experienced by a packet as it propagates through the symmetric system, specifically the delay asymptotic behavior as $N \to \infty$.
We put forward some conjectures about this behavior, supported by heuristic arguments and 
simulation experiments. 
	}



\section{Introduction}
Pairwise information exchange forms the basis of content dissemination in many network applications.
The model considered in this paper is a network (graph) $H$ with finite number $N$ of nodes connected by directed links. Each node creates new packets at the epochs of an independent point process.
We treat packet creation process at each node as an exogenous packet arrival process with finite rate.
The nodes interact pairwise via the network links. For each link there is an independent point process, determining the epochs at which the interactions may occur; the rate of this process is the link capacity. 
During each interaction on a link, a single packet can be transmitted on the link from one node to the other.
Each packet needs to be disseminated to all other nodes; when this occurs, the packet leaves the network.
(In this paper we will use terms packet `dissemination', `propagation', and `replication' interchangeably.)
Under appropriate assumptions, this model is described as a Markov process.

Key questions about the performance of such model include stability (understood as positive recurrence of the corresponding Markov process) and packet delays. In this paper we study the stability and performance of the network
under several communication disciplines, which define the order in which packets are transmitted. 
For a given link $(u,v)$, a packet already available (present) at node $u$, but not yet available at node $v$ is
called {\em useful} (on this link). We consider the following three communication disciplines: $(1)$ the \emph{Random-Useful}, $(2)$ the \emph{Oldest-Useful}, and $(3)$ the \emph{Selfish}. We now describe them.

Absent any application context, a natural discipline is the \emph{Random-Useful}, in which the sending node picks a useful packet uniformly at random.
This discipline was also studied in \cite{massoulie2008rate} for use in peer-to-peer live streaming.

Blockchains are an emerging application which are well-described by our system model; blocks in a blockchain system are analogous to packets in our model.
The \emph{Oldest-Useful} discipline, in which nodes transmit the oldest useful packet, plays a key role in the distributed information verification in blockchain systems (see, \textit{e.g.}, \cite{frolkova2019bitcoin}, \cite{gopalan2020stability}).
The Oldest-Useful discipline is also a natural candidate when one of the goals is to minimize the Age-of-Information (the age in the network of the oldest present packet). Indeed, it is clearly related to much studied disciplines of the Largest-Delay-First (LDF) or Earliest-Deadline-First (EDF) type (see, e.g. \cite{stolyar2003control} for an overview). The disciplines of the latter type are known to be optimal  in terms of stability and Age-of-Information,
in a variety of settings. One of the findings of this paper is that those optimality properties of LDF/EDF disciplines do not always extend to data dissemination networks -- in particular, Oldest-Useful does not necessarily have the maximum stability region for our model. 

The \emph{Selfish} discipline is such that each node prioritizes useful packets that it has created, and sends them in the oldest-first order; other useful packets are sent only when ``own'' packets are not available. 
This discipline is appropriate for systems in which limited or no information is available about not ``own'' packets.

In this paper we show that the Random-Useful discipline achieves the maximum stability region on an arbitrary network $H$.
The main stability results of the paper concern the Oldest-Useful discipline. We show that, in general, it does {\em not} achieve the maximum stability region.
However, for the important special case when the system is symmetric (graph $H$ is complete, all link capacities are equal, and arrival rates to all nodes are equal)
we prove that Oldest-Useful {\em does} provide maximum stability. (This important special case is a good model for, \textit{e.g.}, overlay networks.)
The Selfish discipline is also maximally stable for a symmetric system. 
We also show that Oldest-Useful is maximally stable in networks where there is at most one directed path from any source node to any other node.

Stability properties of a network determine the system throughput. For applications such as peer-to-peer live streaming and peer-to-peer gaming, not only the throughput, but also packet delays matter.
In particular, it is desirable that the expected sojourn time of packets does not differ ``too much'' from those in a system where each packet propagates ``freely,'' with no other packets present in the network.
We refer to the latter system as the \emph{free system}.
To this end, we conjecture (and observe in simulations) that in a symmetric system, as $N\to\infty$, the expected steady-state sojourn time of a packet under the Oldest-Useful and Random-Useful discipliness, normalized by the expected sojourn time of a packet in the corresponding free  system,
converges to a finite constant (depending on the discipline), which we call the {\em limiting slowdown}.

In addition to sojourn times, for a symmetric system, we study the delays 
until a packet is available at $k\in \{1,\ldots,N\}$ out of $N$ nodes. For that we introduce a convenient transformation 
of the system, in which, for each $N$, availability levels $k\in \{1,\ldots,N\}$ are mapped (non-linearly) into points of continuous $[0,1]$ segment; we also rescale time, so that the expected sojourn time in the free system is $1$. 
In the transformed system, packets arrive at point $0$, then move ``right,'' and depart upon reaching point $1$.
The dependence of the expected steady-state time for a packet to reach point $x \in [0,1]$ in the transformed 
system on $x$, we call {\em mean normalized delay profile}. In particular, in the transformed free system, when $N$ is large, a packet essentially moves at constant speed $1$ from point $0$ to $1$; consequently,
the free system mean normalized delay profile converges to the identity function $x$, as $N\to\infty$.
We conjecture, and observe in simulation, that for the Oldest-Useful and Random-Useful disciplines, as $N\to\infty$,
the mean normalized delay profile also converges to a fixed continuous function (depending on the discipline).

To provide some analytic insight into these conjectures about packet delays, we introduce a version of the Random-Useful discipline, 
called the \emph{Reshuffling}, under which a packet with availability $k$, independently of all other packets
is ``reshuffled'' just before any communication epoch in the network -- namely,
it is placed at a set of $k$ nodes chosen uniformly at random.
For a symmetric system under Reshuffling, 
we provide heuristic arguments allowing us to conjecture the explicit linear form of the 
limiting mean normalized delay profile and limiting slowdown. These conjectures for the reshuffling system
are well confirmed by simulation experiments.

\subsection{Related Work}

Several recent works have focused on the propagation of blocks in a blockchain network.
The well-studied model in \cite{pass2017analysis} assumes that every block propagates to all network nodes within a finite deterministic time interval after its arrival.
Papers ~\cite{li2018blockchain, li2019markov} treat block generation as a queueing process but do not consider block propagation in a blockchain network.
Work \cite{sankagiri2021longest} studies only the point-to-point propagation of blocks.
Each of these models ignore the intricacies of pairwise information exchange and instead opt for simpler dynamics with obvious stability regions.
In this paper, we address the nontrivial challenge of taking into account the pairwise information exchange.

The following three papers all study various formulations and restrictions of the pairwise communication model in this paper.

Paper \cite{ioannidis2009optimal} studies a model with a single source of exogenously arriving packets, where network nodes communicate only the newest available packet, whenever such a packet is useful.
While this model is based on pairwise information exchange, the restricted context is such that the system is automatically stable for all positive arrival rates.
We focus on the more general setting where stability is not \textit{a priori} guaranteed.

In \cite{massoulie2008rate} authors study the Random-Useful discipline with a single source of exogenous arrivals, and an arbitrary network topology.
They show that in this setting, the Random-Useful discipline achieves the maximal stability region.
We show that this result extends to general case, with arrivals to more than one node.

Paper \cite{gopalan2020stability} studies the Oldest-Useful discipline.
It obtains lower and upper bounds on the stability region of this discipline for arbitrary network topologies.
However, these bounds are not of the same order: the lower bound is $O(1/\log N)$ of the upper bound, where $N$ is the number of nodes.
In the present paper we use different methods to prove that the Oldest-Useful discipline attains the maximum stability region for symmetric systems
(and also show that, in general, Oldest-Useful is {\em not} maximally stable).

In a setting different from ours, all packets to be disseminated exist at the beginning of the process.
Then there exists a time $T$ at which all packets have 
left the network. The main problem of interest here is to compute moments of $T$ in terms of the network topology.
This setting, also called rumor spreading, has been used to model information propagation in peer-to-peer systems such as Bit-Torrent (see \cite{shah2009gossip} and the references therein).

While our focus is on systems where packets are created exogenously, for completeness we refer the reader to a line of work where pairwise interactions are used to share a fixed number of packets among network nodes that arrive to the system exogenously and leave the system upon collecting a copy of each packet; see, \textit{e.g.},~ \cite{massoulie2005coupon,hajek2011missing}.

\subsection{Contributions of This Paper}

\begin{enumerate}
	\item We extend the result of \cite{massoulie2008rate} to the case where packets arrive at an arbitrary number of nodes (as opposed to just one); namely, we show that Random-Useful is maximally stable in this case. 
	\item Despite it being a natural discipline in various application contexts, we show that the Oldest-Useful discipline in general does not achieve the maximal stability region.
	\item For the special case when the network is symmetric (graph $H$ is complete, all links have equal capacity, and all nodes have equal exogenous arrival rates of new packets), we show that the Oldest-Useful discipline does achieve the maximal stability region.
		This is an important special case as it captures the communication behavior of many overlay networks.
	\item We show that the Oldest-Useful discipline is maximally stable in a network where there is a unique directed path from any source node to any destination node.
	We also show that in a somewhat more special case when the network has a ``tree structure,'' any work-conserving discipline is maximally stable.
	\item We compare the Random-Useful, Oldest-Useful, and Selfish disciplines numerically for symmetric systems.
		We observe that, as desired, the Age-of-Information is smaller under the Oldest-Useful than under the Random-Useful. Somewhat surprisingly, the Selfish discipline achieves both the least Age-of-Information as well as the least packet sojourn time 
		in the system, among the three disciplines.
	\item We make some conjectures about the packet propagation delays in a symmetric system as compared to the free system, which are well-supported by simulation evidence.
		In particular, we conjecture that, as $N\to\infty$, the expected steady-state sojourn time of a packet under the Oldest-Useful and Random-Useful, normalized by the expected sojourn time of a packet in the free  system, converges to a finite constant (depending on the discipline), which we call the {\em limiting slowdown}.
		Moreover, we make a stronger conjecture that for the Oldest-Useful and Random-Useful disciplines, as $N\to\infty$,
the mean normalized delay profile converges to a fixed continuous function (depending on the discipline);
to do that we introduce a novel (non-linear) transformation of the symmetric system.

	\item To gain further insight into the conjectures about packet delays, 
	we introduce a version of the Random-Useful discipline, 
called the \emph{Reshuffling}. For the symmetric system under Reshuffling, we provide heuristic arguments allowing us to conjecture the explicit linear form of the limiting mean normalized delay profile and limiting slowdown. These conjectures for the reshuffling system
are well confirmed by simulation experiments.
	
\end{enumerate}

\subsection{Outline of the Rest of the Paper}

In Section \ref{sec:model}, we introduce the formal system model and the communication disciplines of interest;
we also discuss the model assumptions and how they can be generalized.
Section \ref{sec:random} contains stability results for Random-Useful and Selfish disciplines. 
Section \ref{sec:oldest-first} presents instability and stability results for Oldest-Useful discipline.
Section \ref{sec:single-path} stability results for ``unique-path'' networks and networks having ``tree structure.''
In Section \ref{sec:comparisons}, we present simulation results comparing the performance of Random-Useful, Oldest-Useful, and Selfish disciplines
in symmetric systems.
In Section~\ref{sec:delays}, we state some conjectures regarding the limiting slowdown and limiting mean normalized delay profile in a symmetric network under the Oldest-Useful and Random-Useful disciplines, and provide simulation evidence for them; to formulate these conjectures, we define a free system, which serves as a benchmark.
In Section~\ref{sec:reshuffling}, we introduce the Reshuffling discipline and study it in a symmetric system; in particular, we heuristically derive
the limiting mean normalized delay profile for it, which well confirmed by simulations.
Section~\ref{sec-conclusions} contains brief conclusions.
Appendix~\ref{sec-ru-gen-proof} has the proof of the maximum stability of Random-Useful;
Appendix~\ref{sec-profile-dynamics} contains more technical parts of our analysis of Reshuffling discipline;
Appendix~\ref{sec-supporting-simulation} has some additional supporting simulation results.

\section{Model}
\label{sec:model}

\subsection{Model Definition}
\label{subsec:model}

We consider a network whose structure is given by a finite directed graph $H := (V, E)$, where $V$ and $E$ are the sets of nodes and links, respectively.
Each node $u$ receives an independent flow of exogenous packet arrivals (creations), as a point process of rate $\lambda_u \geq 0$. 
(For those packets node $u$ is the \emph{source}, or \emph{origin}.) 
For each link $(u, v) \in E$, there is a point process of rate $c_{uv}>0$; at the epochs of this process node $u$ can instantly communicate (or transmit) one packet 
to node $v$. The network goal is to disseminate (or propagate, or replicate) each arriving packet to all other nodes. A packet disseminated to all nodes leaves 
the network. 

A packet {\em age} is defined as the time elapsed from the packet exogenous arrival (to its source node) to the current time. One packet is ``older'' or ``younger'' than another if its age is larger or smaller, respectively.
A packet {\em sojourn time} is the time it spends in the network, from exogenous arrival until the departure.
The network current {\em Age-of-Information} (AoI) is the age of the oldest packet in the network.
Throughout the paper $N$ denotes the total number of nodes (the cardinality of $V$), and 
the term \emph{stage $i$ packet}, $i \in \{1,2,\ldots, N-1\}$, refers to a packet available at exactly $i$ nodes.

We now define the communication disciplines considered in this paper. 
We assume that, at any time, for any link $(u,v)$, node $u$ always knows which of the packets already available at $u$ are not yet available at $v$ -- 
we call such packets {\em useful} on link $(u,v)$. (Note that a useful packet on link $(u,v)$ is {\em not necessarily} a packet originating at $u$.)
A communication discipline is a rule specifying, for each link $(u,v)$, which currently useful (on this link) packet is transmitted at this link's communication epoch. (If there are no useful packets on link $(u,v)$ at its communication epoch, then, of course, any packet transmission on this link will not change the system state; thus, we can and will use the convention that no packet transmission occurs.)\\

{\noindent \textbf{Oldest-Useful -- }} Communicate the oldest useful packet (on the link). \\
{\noindent \textbf{Random-Useful --}} Communicate a useful packet (on the link) chosen uniformly at random.\\
{\noindent \textbf{Selfish --}} Communicate the oldest useful packet (on the link), among those originated at $u$. 
If no useful packet originated at $u$ is available, 
communicate another useful packet, chosen according to arbitrary well-defined rule (for example, the oldest, or the youngest, or chosen uniformly at random).\\

Assume now that all {\em packet arrival processes and all link communication processes are mutually independent Poisson.}
Under this assumption, for each of the three communication disciplines, the system evolution can be described by the following continuous-time irreducible Markov chain $X(t)$ with a countable state space.

At time $t$, suppose that $k$ packets are present in the system; we treat any packet which is replicated at all nodes as having departed the system.
We label the packets $1, \ldots, k$, in the order of exogenous arrival, with the oldest packet labeled $1$.
If, at time $t$, a new packet arrives, it is labeled $k+1$.
If, instead, the packet labeled $j$ departs the system at time $t$, we decrement the labels of the packets $j + 1, \ldots, k$ by $1$. For $(u, v) \in E$, we denote by $x_{uv}(t)$ the subset of (labeled) packets with source $u$ which at time $t$  are not available at node $v$ yet.
We define $X(t) = (x_{uv}(t))_{(u, v) \in E}$, which is clearly a countable irreducible continous-time Markov chain. 
(Obviously, $X(t)$ is not the only possible representation of the system evolution as a countable Markov chain.
Other -- equivalent -- Markov representations exist under specific disciplines and/or specific networks.
In fact, we will use equivalent alternative state representations, when it is convenient for analysis.)

\begin{definition}
	The system is called \emph{stable} if the corresponding Markov chain $X(t)$ is positive recurrent or, equivalently, 
	if a stationary distribution exists. (The stationary distribution is necessarily unique.)
	For a given $H$ and link capacities $\{c_{uv}\}$, the network \emph{stablility region} is the set of those 
	arrival rate vectors $\{\lambda_u\}$, for which the system is stable.
\end{definition}

The system stability implies, in particular, that in steady-state the AoI has a proper distribution
and that with positive frequency the network reaches ``empty'' state (with no present packets, i.e. when all arrived packets already left the network). Reaching the empty state
is sometimes referred to as network ``synchronization.''

The following condition is easily seen to be necessary for the network stability, {\em under any discipline} (not just the disciplines considered in this paper):
\beql{eq-cut-cond}
\sum_{u \in S}\lambda_u < \sum_{\substack{u \in S \\ v \not \in S}}c_{uv} ~~~ \mbox{for all $S \subset V$}.
\eeql
Unless this condition holds, there exists an overloaded or exactly critically loaded cut in the network, and the stability cannot hold.
Thus, any discipline under which condition \eqn{eq-cut-cond}
 is also sufficient for stability, achieves the maximum stability region.

\begin{definition}
\label{def-homogen}
	We call a system \emph{complete homogeneous} if $H$ is complete and $c_{uv} = c$ for all $(u, v) \in E$.
	For such networks, without loss of generality, we can and will assume $c = 1$. 
	(This condition can always be achieved by rescaling time.)
	A system will be called \emph{symmetric} if it is complete homogeneous and, in addition, $\lambda_u = \lambda$ for all $u \in V$.
\end{definition}

A complete homogeneous network is a good model, for example, when the pairwise communication occurs on an overlay network, such as in peer-to-peer protocols. Note that, for a symmetric network, the necessary stability condition \eqn{eq-cut-cond} is equivalent to $\lambda < 1$. Also, $\lambda$ is the {\em load} of the symmetric network, because packets arrive at the total rate $\lambda N$, each needs to be transmitted $N-1$ times, and the total capacity of all links in the network is $N(N-1)$; in particular, if the network is stable, then, by symmetry, the utilization of each link (the average rate
at which useful packets are transmitted) is exactly $\lambda$.

\begin{definition}
	We call a system \emph{unique-path} if the graph $H$ is such that it has at most one directed path (not including any node more than once) 
	from any node $u$ to any other node $v$. We will say that a system has {\em tree structure}, if it can be constructed as follows: we take an undirected tree on the set of nodes $V$, and then replace each undirected edge by the pair of directed links in both directions. (Clearly, a network with a tree structure is a special case of
	a unique-path network.)
\end{definition}

\subsection{Discussion of the Model Assumptions and Results' Generalizations}
\label{sec:generalizations}

\subsubsection{Markov Assumptions} 
\label{subsec-markov}

The Markov assumptions, namely that the exogenous arrivals and link communications are Poisson processes, are made to simplify the exposition. They make the system process a countable Markov chain, and also simplify some technical details. All our stability results, in Sections~\ref{sec:random}-\ref{sec:single-path}, easily generalize to far more general arrival and communication processes -- for example renewal --
as long as the system process is Markov, possibly with more complicated, non-countable state space.

\subsubsection{Assumption of Instant Communication over a Link}
\label{subsec-instant}

Our model assumes that packet transmissions across links occur instantly at the epochs of communication point processes. The stability results, in Sections~\ref{sec:random}-\ref{sec:single-path}, still hold if packets require some random transmission time on a link. In this case link capacity $c_{uv}=1/m_{uv}$, where $m_{uv}$ is the mean packet transmission time on link $(u,v)$.
In the special case when the arrival processes are Poisson {\em and} the transmission times happen to be i.i.d. exponentially distributed {\em and} each link uses preempt/resume rule while transmitting its useful packets, the resulting
system process will be {\em exactly same} as in our model with instant communications (and under Markov assumptions); therefore, our stability results hold as is. Under more general assumptions on the arrival processes and transmission time distributions, our stability results easily generalize, again (just as with the generalizations in Section~\ref{subsec-markov}), as long as the system process is Markov, with possibly non-countable state space.
Given these generalizations, it is fair to say that the assumption of instant communication over the links is non-essential.

\subsubsection{Assumption of the Knowledge of Packet Availability at Neighboring Nodes}
\label{subsec-knowledge}

Our model assumes that each node $u$ knows which of its available packets are useful on each of its
outgoing links $(u,v)$. This assumption is a slight idealization of the realistic situation where the time/resource required for a node to inform its neighbors about receiving a new packet is negligible compared to the time required to actually transmit a packet ``payload'' over a link. And as we explained in Section~\ref{subsec-instant}, our results do apply to 
the system with positive expected transmission times $m_{uv}=1/c_{uv}$ over links.

\subsection{Basic Notation}
We use the notations $A \subseteq B$ and $A \subset B$ to denote subsets and strict subsets, respectively.
For a set $S$, $|S|$ denotes its cardinality. Abbreviation {\em a.s.} means {\em almost surely}; $\Rightarrow$
signifies convergence in distribution; $\mathrm{Poisson}(\psi)$ means {\em Poisson distribution with mean $\psi$}.

\section{Random-Useful and Selfish Disciplines' Stability Results}
\label{sec:random}

\subsection{Random-Useful Discipline Maximum Stability}
\label{subsec:random}
In this section we discuss the maximum stability of the Random-Useful discipline, for arbitrary network topologies.

The following previous result proved that, under Random-Useful, condition \eqn{eq-cut-cond} is not only necessary, but also sufficient for stability, for an arbitrary network with a single source.

\begin{theorem}[theorem 3 in 
\cite{massoulie2008rate}]
\label{th-massoulie}
Under Random-Useful discipline, any network $H$ with a single source node $s$ is stable iff condition \eqn{eq-cut-cond} holds.
(Thus, Random-Useful has maximum stability region for single-source networks.)
\end{theorem}

Examination of the proof of this result in \cite{massoulie2008rate} shows that it can be extended to multi-source networks as well, so we obtain the
following more general result.

\begin{theorem}
\label{th-ru-generalization}
Under Random-Useful discipline, any network $H$ is stable iff condition \eqn{eq-cut-cond} holds.
(Thus, Random-Useful has maximum stability region.)
\end{theorem}

The proof of Theorem~\ref{th-ru-generalization} is an adjustment of the proof of theorem 3 in \cite{massoulie2008rate},
which is given in Appendix~\ref{sec-ru-gen-proof}.

\subsection{Selfish Discipline Stability in a Complete Homogeneous Network}
\label{sec-opportune}

The following fact is rather simple. We present it for completeness.

\begin{lemma}
	Consider a complete homogeneous network (Definition~\ref{def-homogen}) under the Selfish discipline.
	(Recall that, without loss of generality, all link capacities are equal to $1$.)
	Then the process $X(t)$ is stable if $\lambda_u < 1$ for all nodes $u$.
	If, in addition, the network is symmetric ($\lambda_u = \lambda$ for all nodes $u$), it is stable iff $\lambda < 1$.	
	\label{lem:selfish}
\end{lemma}

\begin{proof}
The network is stable if each node $u$ communicates {\em only} its ``own'' packets -- those originating at the node itself. 
Indeed, if we denote by $Q_{uv}(t)=|x_{uv}(t)|$ the number of packets originating at node $u$ and not available at node $v$ yet, then the evolution of $Q_{uv}(t)$ is that of a sub-critically loaded M/M/1 queue. (Processes $Q_{uv}(\cdot)$
for different links $(u,v)$ are not independent.) 
Therefore, for any fixed initial state 
of the system, each $Q_{uv}(t)$ is stochastically upper bounded: for any $\epsilon>0$, there exists $C>0$, such that
$\pr\{Q_{uv}(t) > C\} \le \epsilon$, uniformly in $t$. Then $\sum_{(u,v)} Q_{uv}(t)=\sum_{(u,v)} |x_{uv}(t)|$, 
which is an upper bound on the total number of packets in the network, 
is also stochastically upper bounded uniformly in $t$. This implies stability.

Now, if, in addition, nodes communicate other useful packets (when own ones are unavailable), a straightforward coupling can be used to show that the total number of packets in the network at any time can only be smaller or equal.
This implies stability under Selfish discipline.

Finally, recall that, for a symmetric network, the necessary stability condition \eqn{eq-cut-cond} is equivalent to $\lambda < 1$. 
\end{proof}

\section{Oldest-Useful Discipline Stability Results}
\label{sec:oldest-first}

In this section, we discuss the stability of the Oldest-Useful discipline.
We first give an example of a network in which the Oldest-Useful discipline does not achieve the maximal stability region.
Next, we show that a complete homogeneous network is stable under Oldest-Useful when the arrival rates at all nodes are less than the (equal) link capacities; this, in particular, implies the maximum stability of symmetric networks.

\subsection{Oldest-Useful Does Not Achieve Maximum Stability}

Recall that condition \eqn{eq-cut-cond}  is necessary for stability.
We now show an example demonstrating that condition \eqn{eq-cut-cond} is {\em not} sufficient for stability under the Oldest-Useful discipline;  
this will prove that, in general, Oldest-Useful does {\em not} achieve maximum stability region.
We provide a detailed sketch of the argument here, which highlights the key ``reason'' for instability; turning this sketch into a detailed rigorous proof is straightforward.

Consider the network of $3$ nodes in Figure \ref{fig:counterexample}, with the Oldest-Useful discipline.
Communications on any link are attempted at rate $1$.
\begin{figure}[h!]
\center
   	\begin{tikzpicture}
   		\draw[->, >= stealth] (1.732*0.21, 0.21) -- (1.732*0.8, 0.8) node[midway, above] {$1$};
   		\draw[->, >= stealth] (1.732*0.21, -0.21) -- (1.732*0.8, -0.8) node[midway, below] {$1$};
   		\draw[->, >= stealth] (1.632, 0.6) -- (1.632, -0.6) node[midway, left] {$1$};  
   		\draw[<-, >= stealth] (1.832, 0.6) -- (1.832, -0.6) node[midway, right] {$1$};
		\draw[fill=white] (0,0) circle (10pt);
		\draw[fill=white] (1.732, 1) circle (10pt);
		\draw[fill=white] (1.732, -1) circle (10pt);
		\node at (0, 0) {$1$};
		\node at (1.732, 1) {$2$};
		\node at (1.732, -1) {$3$};
		
		\node at (-2.25, 0) {$\lambda_1 = 2-\varepsilon$};
		\node at (3.732, 1) {$\lambda_2 = 0$};
		\node at (3.732, -1) {$\lambda_3 = 0$};
		\draw[->, >= stealth] (-1.3, 0) -- (-0.5, 0);
		\draw[->, >= stealth] (3.132, 1) --(2.232, 1);
		\draw[->, >= stealth] (3.132, -1) --(2.232, -1);
	\end{tikzpicture}
	\caption{An example of an Oldest-Useful network which does not achieve the maximal stability region.}
	\label{fig:counterexample}
\end{figure}
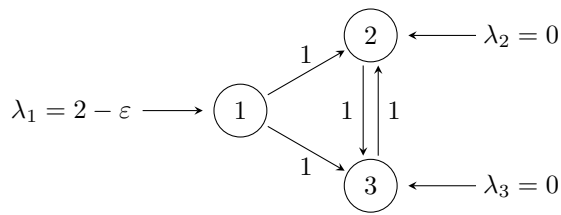

We set $\lambda_1 = 2 - \varepsilon, \lambda_2 = \lambda_3 = 0$, where $\varepsilon >0$ is small and will be specified later. 

The basic intuition for this system instability under Oldest-Useful is as follows. All exogenous packet arrivals occur at node $1$. All these packets need to cross the cut from node $1$ to nodes $2$ and $3$, i.e. each of them will be communicated  on link $(1,2)$ or link $(1,3)$ {\em or both}. As we will see shortly, a non-zero fraction of the packets will be communicated {\em on both} links $(1,2)$ and $(1,3)$. But, the total capacity of these two links is $2$. Therefore, {\em distinct} packets will cross the cut from node $1$ to nodes $2$ and $3$ at the rate which is strictly less than $2$. 
(If same packet is transmitted on  both links $(1,2)$ and $(1,3)$, then still only one distinct packet crossed the cut.) 
This means that if 
$\lambda_1 = 2 - \varepsilon$ is sufficiently close to $2$, the number of packets at node $1$ will grow to infinity. 
We now proceed with details.

For this particular system, a state of the Markov process, can be encoded as follows. 
(This state representation is equivalent to $X(t)$.)
We will use examples. State $(2,2,2,1,1,1,1)$ describes the network, which currently has 7 packets. They all arrived at node $1$ (because this is the only node where exogenous arrivals occur). The packets are listed in the order of their exogenous arrivals, starting from the oldest. Each symbol '2' corresponds to a packet present in nodes $1$ and $2$, but not $3$. Each symbol '1' corresponds to a packet present in node $1$ only. Note that, by the Oldest-Useful definition, all 2-packets must be at the beginning of the sequence, because each of them must be ``older'' than any 1-packet. Another example of a state is $(3,3,1,1,1)$;
here symbol '3' corresponds to packets present in nodes $1$ and $3$, but not $2$. Note that, by the Oldest-Useful definition, we can never have a state which has both 2-packets and 3-packets. (Because, for example, a state
$(2,3,2,1,1)$ would mean that a ``younger'' packet was communicated on link $(1,3)$ when node $1$ had available an ``older'' packet, which was useful at node $3$.) Let us refer to 2-packets and 3-packets as ``in-transit'' packets. States, which have no in-transit packets, we will call ``renewal.''

Consider an initial state, which is renewal and has very large fixed number of packets; that is a state like 
$(1,1,1,1,1,\ldots, 1,1)$. Starting this renewal state, consider the ``projected'' process that only tracks the in-transit packets, while ``ignoring'' the very large number of 1-packets, appended at the end. In other words, if a state is $(2,2,2,1,1,\ldots,1,1)$, we only consider its projection $(2,2,2)$. The empty state of the projected process is a renewal state of actual process. It is easy to see that the projected process is stable, that is the mean inter-renewal time is finite. 
Indeed, consider a projected state consists of say 2-packets only (recall that it cannot contain 2-packets and 3-packets simultaneously), like $(2,2,2)$. Then the rate at which 2-packets are added (due to communications on link $(1,2)$)   is $1$, while the rate at which 2-packets are eliminated (due to communications on links $(1,3)$ and $(2,3)$) is $2$.
Therefore, the number of packets in a projected state has negative drift, which guarantees stability.

Now, we claim that, when the projected process is in steady-state, new in-transit packets are created at the rate $2-\delta < 2$. Indeed, a new in-transit packet is created when and only when a 1-packet is communicated on either link $(1,2)$ or $(1,3)$. When say 2-packet is communicated on link $(1,3)$, no new in-transit packet is created (and in fact this 2-packet is eliminated); similarly, when a 3-packet is communicated on link $(1,2)$, no new in-transit packet is created. 
Note that after any new in-transit packet, say 2-packet, is created (due to 1-packet being communicated on link $(1,2)$), with probability at least $1/4$ the next packet communicated by node 1 will be that newly created 2-packet 
communicated on link $(1,3)$. (Because with probability $1/4$ the next communication in the entire system will occur on link $(1,3)$, upon which, necessarily, that  2-packet will be sent.) This means that, in steady-state, a non-zero fraction of packets communicated on links $(1,2)$ and $(1,3)$ will be in-transit packets. But the total rate of all communications on links $(1,2)$ and $(1,3)$ is the total capacity of those links, which is $2$. Thus, that rate at which new in-transit packets are created is strictly less than $2$, which proves the claim.

Now recall that we consider the initial state $(1,1,1,1,1,\ldots, 1,1)$ of the actual process. Its evolution is such that the new 1-packets are added at rate $2-\varepsilon$, while 1-packets are ``converted'' to in-transit packets at rate $2-\delta$.
We conclude that, if $0< \varepsilon < \delta$, the number of 1-packets will have a strictly positive drift, and will tend to infinity with probability $1$. Thus, this network is unstable, despite satisfying the condition \eqn{eq-cut-cond}.

\begin{rem}
\label{rem-ru-diversity}
From Theorem~\ref{th-ru-generalization} we know that the above specific system (as well as any system 
satisfying \eqn{eq-cut-cond}) must be stable under the Random-Useful discipline. However, for this specific system,
there is a ``simple reason'' for Random-Useful stability. Indeed, if the number of packets at node $1$ is very large, then under the Random-Useful only a very small fraction of packets being communicated across the cut will use both links $(1,2)$ and $(1,3)$. Thus, the ``diversity'' feature of Random-Useful prevents it, in this system, from ``wasting'' a non-zero fraction of the capacity of the cut from node $1$ to nodes $2$ and $3$. However, it is important to note that, for a general network model that we consider in this paper, the Random-Useful diversity feature alone does {\em not} easily lead to its maximum stability. This requires a relatively involved proof of Theorem~\ref{th-ru-generalization}.
\end{rem}

\subsection{Oldest-Useful Stability in a Complete Homogeneous System} 
\label{sec-ou-stabil}

Our main stability result for the Oldest-Useful discipline is as follows.
\begin{theorem}
	Consider a complete homogeneous network (Definition~\ref{def-homogen}) under the Oldest-Useful discipline.
	(Recall that, without loss of generality, all link capacities are equal to $1$.)
	Then:\\
	(i) Process $X(t)$ is stable if $\lambda_u < 1$ for all nodes $u$.\\
	(ii) If, in addition, the network is symmetric ($\lambda_u = \lambda$ for all nodes $u$), $X(t)$ stable iff $\lambda < 1$.	
	\label{thm:main-result}
\end{theorem}

Statement (ii) is a corollary of (i), 
because for a symmetric network the necessary stability condition \eqn{eq-cut-cond} is equivalent to $\lambda < 1$.
The rest of Section~\ref{sec-ou-stabil} contains the proof of Theorem~\ref{thm:main-result}(i). 
We remark that Theorem~\ref{thm:main-result}(i) holds for a more general case, when all $\lambda_u < 1$ and all link capacities {\em are greater than or equal $1$.} Same proof works with straightforward adjustments.

\subsubsection{Proof Highlights}
\label{subsubsec:proof-highlights}
We use the \emph{fluid limit} technique (see, \textit{e.g.}, \cite{rybko1992ergodicity, dai1995positive, stolyar1995stability,Bramson-book}).
 This involves considering a sequence of processes, scaled in both space and time, which converges to a limiting process
with trajectories called fluid limits. Then it suffices to show that,  there exists a time $T > 0$ such that,
for any fluid limit with ``norm'' $1$ of the initial state, the state norm reaches $0$ by time $T$ and remains at $0$ thereafter.
To do this, roughly speaking, we use the arrival time of the currently oldest packet in the system as a Lyapunov function.
We show that, for any fluid limit, this arrival time progresses at rate strictly greater than $1$, as long as it is less than the current time $t$;
thus it reaches $t$ by a finite time $T$, and then coincides with $t$ thereafter.

The above Lyapunov function and the argument are related to those used in ~\cite{stolyar2003control}. There is a substantial difference though. 
In ~\cite{stolyar2003control} the packet flows of each type follow their own deterministic routes through the network. For the model in this paper this is not the case -- packets with the same origin propagate through the network in a random fashion. As a result, our argument is different from that in ~\cite{stolyar2003control}.

\subsubsection{Fluid Limits}
Define the ``norm'' $||X(t)|| = \sum_{(u, v) \in E} |x_{uv}(t)|$, where $|x_{uv}(t)|$ is the cardinality of $x_{uv}(t)$.
Consider a sequence of processes $(X^n(\cdot))_n$ indexed by some sequence $n \to \infty$ such that the norms of the initial states satisfy $||X^n(0)|| = n$.
To prove Theorem \ref{thm:main-result}, it suffices to show \cite{rybko1992ergodicity, dai1995positive, stolyar1995stability,Bramson-book} that for any such sequence and for some universal time $T > 0$, 
\begin{equation}
	\frac{1}{n}\mathbb{E}[||X^n(nT)||] \to 0 \text{ as } n \to \infty.
	\label{eq:mean-limit}
\end{equation}
Standard arguments \cite{rybko1992ergodicity, dai1995positive, stolyar1995stability,Bramson-book} apply to see that
the family of random variables $\{||X^n(nT)||/n\}$ is uniformly integrable and, therefore,
to prove \eqref{eq:mean-limit}, it suffices to show that $||X^n(nT)||/n \to 0 ~ \mathrm{a.s.}$. To prove the latter
it in turn suffices to show that for any subsequence of $\{n\}$, there exists a further subsequence $\{m\}$ along which
\begin{equation}
	\frac{1}{m} ||X^m(mT)|| \to 0 \quad \mathrm{a.s.}
	\label{eq:norm-cond}
\end{equation}
Toward this goal, consider some subsequence $\{n\}$.

 We will interpret an exogenous arrival of a packet at node $u$ as $N-1$ copies of this packet simultaneously joining 
$N-1$ {\em virtual queues}, one per each link $(u,v)$ with $v \ne u$. A copy joining an $(u,v)$ virtual queue we will call a \emph{$(u,v)$-packet}.
The process state component $x_{uv}(t)$ can be thought of as describing the content of the virtual queue of $(u,v)$-packets that arrived by time $t$
and are still undelivered to $v$ by $t$. 
We note that a $(u,v)$-packet departure from the $(u,v)$ virtual queue is {\em not} necessarily
due to the direct delivery of the corresponding actual packet from node $u$ to node $v$ -- it is possible that the actual packet is delivered to node $v$ by actual transmission on another link $(w,v)$. That's why we call these $(u,v)$-queues virtual. (We also want to emphasize that the virtual queues are not part of the Oldest-Useful discipline definition,
and are not directly used by the algorithm. They just represent our view of the system evolution, which is convenient for analysis.) 

It is easy to observe that Oldest-Useful discipline has the following property. At any time, {\em if there are two packets, $p_1$ and $p_2$, originating (exogenously arrived) at the same node, and $p_2$ is ``younger'' (arrived in the system later) than $p_1$, then $V_2 \subseteq V_1$,} where $V_i$ is the set of nodes where $p_i$ is present.
This, in turn, implies that {\em $(u,v)$-packets depart the $(u,v)$ virtual queue strictly in the order of their age -- oldest first.}

To prove \eqref{eq:norm-cond}, we define the following quantities related to the evolution of the system for all $(u, v) \in E$ and $w \in V$:

\begin{itemize}
	\item $F_{uv}(t)$ is the number of $(u,v)$-packets which have arrived to the system by time $t$.
	\item $\hat{F}_{uv}(t)$ is the number of $(u,v)$-packets, left $(u,v)$-queue by time $t$.
	\item $\hat{F}_{uv, w}(t)$ is the number of $(u,v)$-packets, left $(u,v)$-queue by time $t$ due to their transmission on link $(w, v)$.
	\item $S_{uv}(t)$ is the number of communication epochs on link $(u, v)$ in the interval $[0, t]$.
\end{itemize}

Recall that the definition of our Markov process state is such that only the order of packets from ``oldest'' to ``youngest'' matters, not their actual times of arrival into the system. Therefore, we can and will 
adopt the following convention about the packets present in the system at time $t=0$. (It will 
allow us to define  functions $F_{uv}(t)$ for all real times $t$, which will
be convenient.)
The convention is
that if there are $\ell$ packets in the system at time $t = 0$, they arrived exogenously at non-positive time instants $-\ell + 1, -\ell + 2, \ldots, -1, 0$, ``in the past.''
Then, $F_{uv}(t)$ is defined for $t \geq -\ell+1$ and we use the convention that $F_{uv}(t) = 0$ for $t < -\ell + 1$ for all $(u, v) \in E$.
Functions $\hat{F}_{uv}(t)$ and $\hat{F}_{uv, w}(t)$ are defined for $t \geq 0$  with $\hat{F}_{uv}(0) = \hat{F}_{uv, w}(0) = 0$  for all $(u, v) \in E, w \in V$.
We note that functions $F_{uv}(t)$, $\hat{F}_{uv}(t)$ and $\hat{F}_{uv, w}(t)$ are non-decreasing in $t$, for all $(u, v) \in E, w \in V$.
Functions $S_{uv}(t)$ are defined for $t \geq 0$ and are also non-decreasing, for all $(u, v) \in E$.

The process $(F_{uv}(\cdot), S_{uv}(\cdot))_{(u, v) \in E}$ uniquely determines the process $X(\cdot)$.
Recall that the ``norm'' of $X(t)$ is the total number of undelivered $(u,v)$-packets (for all $(u,v)$) in the system at time $t \geq 0$: 
$$||X(t)|| = \sum_{(u, v) \in E}F_{uv}(t) - \hat{F}_{uv}(t).$$

We use the superscript $n$ for the quantities pertaining specifically to the quantities related to $X^n(\cdot)$.
For any $n \geq 1$, any $(u, v) \in E$, and any $w \in V$, let $$f_{uv}^n(t) := \frac{1}{n}F^n_{uv}(nt),$$ $$\hat{f}^n_{uv}(t) := \frac{1}{n}\hat{F}^n_{uv}(nt),$$ $$\hat{f}^n_{uv, w}(t) := \frac{1}{n}\hat{F}^n_{uv, w}(nt),$$ and $$s^n_{uv}(t) := \frac{1}{n}S_{uv}(nt).$$
Note that for all $n \geq 1$, $(u,v) \in E, w \in V$, $f^n_{uv}(t) = 0$ for $t \leq -1$, $\hat{f}^n_{uv}(0) = \hat{f}^n_{uv, w}(0) = 0$, and $\sum_{(u, v) \in E}f_{uv}^n(0) = 1$.

Denote by $\Xi_{uv}^n(t)$ the arrival time of the oldest undelivered $(u,v)$-packet at time $t$.
If there is no such packet, set $\Xi_{uv}^n(t) = t$.
Define $$\tau_{uv}^n(t) := \frac{1}{n}\Xi^n_{uv}(nt)$$ and $$\tau^n(t) := \min_{(u, v) \in E}\tau_{uv}^n(t).$$
Note that $\tau_{uv}^n(0) \ge -1$, and then $\tau^n(0) \ge -1$.

Let $(\Omega, \mathcal{F}, \mathbb{P})$ be a probability space of independent driving Poisson arrival processes to the nodes and Poisson service processes on the communication links.
We assume that the processes for all $n$ are constructed on this space, using the same set of driving Poisson processes.
We note that the driving processes satisfy the functional strong law of large numbers: almost surely,
\begin{equation}
	f^n_{uv}(t) - f^n_{uv}(0) \to \lambda_u t \text{ and } s^n_{uv}(t) - s^n_{uv}(0) \to t, ~~\text{$t\ge 0$, u.o.c. as } n \to \infty \quad \forall (u, v) \in E.
	\label{eq:FSLLN}
\end{equation}

\begin{lemma}
	The subsequence $\{n\}$ (fixed earlier) has a further subsequence $\{m\}$ such that $\liminf_{m\to\infty}\tau^m(1) \geq 0$ almost surely.\
	\label{lem:tau(1)}
\end{lemma}
\begin{proof}
Let us refer to $(u,v)$-packets present in the system at (``initial'') time $t=0$ as initial $(u,v)$-packets. 
By the nature of Oldest-Useful discipline, if there is at least one initial packet in any of the virtual queues at time $t$, then 
for any process state at $t$, the time until the next initial packet is delivered (and removed from its virtual queue) is stochastically upper bounded by an exponential random variable with mean $1$. Therefore, the time $\Delta^n$ at which the network delivers all initial packets is stochastically upper bounded by the sum of $n$  i.i.d. exponential mean-1 random variables, where recall $n$ is the total number of initial packets. For the rescaled processes, we see that, for any $\varepsilon>0$, 
$$
\pr\{\tau^n(1) > -\varepsilon\} \to 1, ~~ n\to\infty.
$$
By choosing the further subsequence $\{m\}$ to grow sufficiently fast, we obtain the lemma statement. We omit $\varepsilon/\delta$
formalities.
\end{proof}

From this point, we fix a further subsequence along which $\liminf_{n\to\infty} \tau^n(1) \geq 0$ a.s., as in Lemma \ref{lem:tau(1)}.

\begin{definition}
	A collection of functions $$(f_{uv}(\cdot), \hat{f}_{uv}(\cdot), \hat{f}_{uv, w}(\cdot), s_{uv}(\cdot), \tau_{uv}(\cdot), \tau(\cdot))_{(u, v) \in E, w \in V}$$ is called a \emph{fluid limit} if there exists a sequence of process sample paths $$(f^m_{uv}(\cdot), \hat{f}^m_{uv}(\cdot), \hat{f}^m_{uv, w}(\cdot), s^m_{uv}(\cdot), \tau^m_{uv}(\cdot), \tau^m(\cdot))_{m \in \mathbb{N}, (u, v) \in E, w \in V}$$ such that:
	\begin{enumerate}
		\item Property \eqref{eq:FSLLN} holds.
		\item All of the functions $f_{uv}, \hat{f}_{uv}, \hat{f}_{uv, w}, s_{uv}$ are Lipschitz non-decreasing, and the functions $\tau_{uv}, \tau$ are non-decreasing, for all $(u, v) \in E, w \in V$.
		\item The following convergences hold for all $(u, v) \in E, w \in V$, uniformly on the compact subsets:
			\begin{align}
				& \lim_{m\to\infty} f^m_{uv}(t) \to f_{uv}(t)
				\\
				& \lim_{m\to\infty} \hat{f}^m_{uv}(t) \to \hat{f}_{uv}(t)
				\\
				& \lim_{m\to\infty} \hat{f}^m_{uv, w}(t) \to \hat{f}_{uv, w}(t)
				\\
				& \lim_{m\to\infty} s^m_{uv}(t) \to s_{uv}(t)
			\end{align}
		
		The following convergences hold at all points $t \geq 0$ of continuity of $\tau_{uv}, \tau$ for all $(u, v) \in E$:
			\begin{align}
				& \lim_{m\to\infty} \tau^m_{uv}(t) \to \tau_{uv}(t)
				\\
				& \lim_{m\to\infty} \tau^m(t) \to \tau(t)	
			\end{align}
		\item $\tau(1) \geq 0$.
	\end{enumerate}
	\label{def:fluid-limit}
\end{definition}

\begin{proposition}
	Almost surely, every subsequence of the subsequence fixed above has a further subsequence $\{m\}$ such that the convergence of $$(f^m_{uv}(\cdot), \hat{f}^m_{uv}(\cdot), \hat{f}^m_{uv, w}(\cdot), s^m_{uv}(\cdot), \tau^m_{uv}(\cdot), \tau^m(\cdot))_{(u, v) \in E, w \in V}$$ to a fluid limit $$(f_{uv}(\cdot), \hat{f}_{uv}(\cdot), \hat{f}_{uv, w}(\cdot), s_{ij}(\cdot), \tau_{uv}(\cdot), \tau(\cdot))_{(u, v) \in E, w \in V}$$ holds in the sense specified in Definition \ref{def:fluid-limit}.
\end{proposition}
\begin{proof}
	The proof uses the functional strong law of large numbers \eqref{eq:FSLLN} and standard arguments, see, \textit{e.g.}~\cite{rybko1992ergodicity, dai1995positive, stolyar1995stability, rybko2002stability}.
	\end{proof}

\begin{proposition}
	A fluid limit satisfies the following properties for all $(u, v) \in E, w \in V$,  $t \geq 0$, $0 \le t_1 \le t_2$:
	\begin{align}
		& f_{uv}(t) - f_{uv}(0) = \lambda_u t
		\label{eq:FL-P-1}
		\\
		& s_{uv}(t) - s_{uv}(0) = t
		\label{eq:FL-P-2}
		\\
		& \hat{f}_{uv, w}(t_2) - \hat{f}_{uv, w}(t_1) \leq (t_2 - t_1)
		\label{eq:FL-P-4}
		\\
		& \hat{f}_{uv}(t_2) - \hat{f}_{uv}(t_1) = \sum_{w \in V} \hat{f}_{uv, w}(t_2) - \hat{f}_{uv, w}(t_1)
		\label{eq:FL-P-5}
		\\
		& \tau(t) = \min_{(u, v) \in E}\tau_{uv}(t) 
		\label{eq:FL-P-6}
		\\
		& \tau(t) \leq \inf_{\substack{s \leq t \\ (u, v) \in E}}\{s: f_{uv}(s) - \hat{f}_{uv}(t) > 0\} 
		\label{eq:FL-P-7}
	\end{align}
	\label{prop:FL-properties}
\end{proposition}
\begin{proof}
	Properties \eqref{eq:FL-P-1} and \eqref{eq:FL-P-2} are part of the definition of a fluid limit.
	Property \eqref{eq:FL-P-4} follow from the strong law of large numbers.
	Properties \eqref{eq:FL-P-5}, \eqref{eq:FL-P-6} and \eqref{eq:FL-P-7} follow from the analogous properties for the pre-limit trajectories.
\end{proof}

For $(u, v) \in E, w \in V$, denote by $\mu_{uv, w}(t) := \frac{\mathrm{d}}{\mathrm{d}t} \hat{f}_{uv, w}(t)$, whenever this derivative exists.
The quantity $\mu_{uv, w}(t)$ can be thought of as the ``instantaneous service rate for $(u,v)$-fluid on link $(w, v)$'' at time $t$.
Obviously, $\sum_{u \in V}\mu_{uv, w}(t) \leq 1$ at any time $t \geq 0$ such that the requisite derivatives exist.

\begin{definition}
	A time $t\ge 0$ is \emph{regular} if the derivatives with respect to time $\tau'(t), \tau'_{uv}(t), \mu_{uv, w}(t)$ exist for all $(u, v) \in E$ and $w \in V$.	
\end{definition}

Since $\tau(t), \tau_{uv}(t), \hat{f}_{uv, w}(t)$ are non-decreasing in $t$ for all $(u, v) \in E, w \in V$, it follows 
that almost all times $t \geq 0$ are regular.

\subsubsection{Stability of Fluid Limits}

\begin{definition}
	Consider a fluid limit.
	We say that $(u, v) \in E$ is a \emph{critical link at time $t \geq 0$} if $\tau_{uv}(t) = \tau(t)$. 
	We denote by $L(t) \subseteq E$ the subset of critical links at time $t$.
	\label{def:critical-link}
\end{definition}

\begin{definition}
	Consider a fluid limit.
	Suppose that $t \geq 0$ is a regular time.
	We say that the link $(w, v)$ \emph{does not serve $(u,v)$-fluid at time $t$} if $\mu_{uv, w}(t) = 0$.	
	\label{def:does-not-serve-ij-fluid}
\end{definition}

In the proofs of the following Lemmas \ref{lem:tau-deriv-estimate} and \ref{lem:critical-links-busy}, along with a given fluid limit, we consider a fixed defining sequence of pre-limit sample paths $\{m\}$ converging to the fluid limit as specified in Definition \ref{def:fluid-limit}.

\begin{lemma}
	Consider a fluid limit.
	Let $t \geq 1$ be a regular time and let $W \subseteq V$.
	Then $$\tau'_{uv}(t) \geq \frac{1}{\lambda_u}\sum_{w \in W}\mu_{uv, w}(t).$$
	\label{lem:tau-deriv-estimate}
\end{lemma}
\begin{proof}
	Recall that we index the defining sequence by $m$.	Also recall that, under Oldest-Useful discipline,
	if a link $(w,v)$ serves a $(u,v)$-packet, it is the oldest $(u,v)$-packet. 
	Then, for a small enough $\delta>0$, and all sufficiently large $m$, we must have
	$$
	f_{uv}^m(\tau_{uv}^m(t)+\delta) - f_{uv}^m(\tau_{uv}^m(t)) \ge
	\sum_{w \in W} [\hat f_{uv,w}^m(t+\delta) - \hat f_{uv,w}^m(t)]. 
	$$
	Taking the limit in $m$ and then sing the fact that $t$ is regular, we obtain the result.
\end{proof}

\begin{lemma}
	Consider a fluid limit.
	Let $t \geq 1$ be a regular time such that $\tau(t) < t$.
	Suppose that $(u, v) \in L(t)$.
	Then link $(u, v)$ does not serve any $(w,v)$-fluid such that $\tau_{wv}(t) > \tau(t)$, 
	and $\sum_{(w, v) \in L(t)}\mu_{wv, u}(t) = 1$.
	\label{lem:critical-links-busy}
\end{lemma}

\begin{proof}
	If $\tau_{wv}(t) > \tau(t)$, then for large enough $m$, $\tau^m_{wv}(t) > \tau^m_{uv}(t)$ and for any sufficiently small re-scaled interval $(t, t + \delta)$, link $(u, v)$ does not deliver $(w,v)$-packets, because it has older
	$(u, v)$-packets available for transmission.
	Thus, in the fluid limit, link $(u, v)$ does not serve $(w,v)$-fluid at time $t$.

	Since $\tau_{uv}(t) = \tau(t) < t$, then for sufficiently large $m$, there are always undelivered $(u,v)$-packets at each epoch of $s_{uv}^m(\cdot)$ during a small enough interval $(t, t + \varepsilon)$.
	Thus, a packet is transmitted from $u$ to $v$ each epoch of $s_{uv}^m(\cdot)$ during the interval $(t, t + \varepsilon)$ and the equality $\sum_{(w, v) \in L(t)}\mu_{wv, u}(t) = 1$ follows. 
\end{proof}

\begin{lemma}
	There exists $T > 0$ such that for any fluid limit, $\tau(t) = t$ for all $t \geq T$.
	\label{lem:tau(T)}
\end{lemma}
\begin{proof}

Note that $0 \leq \tau(1) \leq 1$.
Recall that $\tau(\cdot)$ is non-decreasing and that non-regular times form a (Lebesgue) null set.
Fix some regular time $t \geq 1$ such that $\tau(t) < t$.
By Lemma \ref{lem:tau-deriv-estimate},
for any $(u, v) \in L(t)$,
$$
\lambda_u \tau'_{uv}(t) \geq \sum_{(w,v) \in L(t)} \mu_{uv, w}(t),
$$
and then 
$$
 \sum_{(u,v) \in L(t)} \lambda_u \tau'_{u,v}(t) \geq \sum_{(u,v) \in L(t)} \sum_{(w,v) \in L(t)} \mu_{uv, w}(t)
 = \sum_{(w,v) \in L(t)} \sum_{(u,v) \in L(t)} \mu_{uv, w}(t) =|L(t)|;
$$
the last equality is because $\sum_{(u,v) \in L(t)} \mu_{uv, w}(t) =1$ by the last property in Lemma~\ref{lem:critical-links-busy}.
Recall that $\tau'_{uv}(t) = \tau'(t)$ for each $(u,v) \in L(t)$, so that
$$
 \sum_{(u,v) \in L(t)} \lambda_u \tau'_{uv}(t) = \tau'(t)  \sum_{(u,v) \in L(t)} \lambda_u \le \tau'(t) |L(t)| \max_{u} \lambda_u.
$$
We conclude that
$$
\tau'(t) \geq \beta \doteq \frac{1}{\max_{u} \lambda_u} > 1.
$$

 Recall that $\tau(t)$ is non-decreasing. (It is not hard to see that $\tau(t)$ is Lipschitz. 
 But, we will only use the fact that it is non-decreasing. So, formally speaking, in this proof $\tau(t)$ 
need not be absolutely continuous or even continuous.)
Then in any time interval $[s_1,s_2]$, $s_1 \ge 1$, in which $\tau(t) < t$, the increment of $\tau(t)$ is lower bounded as follows:
$$
\tau(s_2) - \tau(s_1) \ge \int_{s_1}^{s_2} \tau'(t) dt \ge \beta (s_2-s_1).
$$
Recall that, at time $1$, $1-\tau(1) \le 2$.
We conclude that the condition $\tau(t) = t$ is reached at some time $s \le T=2/(\beta-1)$, and then holds for all $t \ge s$. 
\end{proof}

We can now complete the proof of Theorem~\ref{thm:main-result}(i) by verifying
property \eqref{eq:norm-cond}.
Note that any fluid limit is such that $t = \tau(t)$, and then $\sum_{(u, v) \in E} [f_{uv}(t) - \hat{f}_{uv}(t, t)] = 0$,
 for all $t \ge T$. For any subsequence $\{n'\}$ of a given sequence $\{n\}$, we have shown the existence of a further subsequence $\{m\}$, for
which the following property holds. W.p.1, any subsequence $\{m'\}$ of $\{m\}$, has a further subsequence $\{m''\}$ along which
$$\lim_{m''\to\infty}\frac{1}{m''}||X^{m''} (m'' T)|| = \sum_{(u, v) \in E} [f_{uv}(T) - \hat{f}_{uv}(T)] = 0$$
for some fluid limit. But this means that, w.p.1, $\lim_{m\to\infty}\frac{1}{m}||X^{m} (m T)||=0$ for the subsequence $\{m\}$ itself,
thus verifying \eqref{eq:norm-cond}. The proof of Theorem~\ref{thm:main-result}(i) is complete.

\section{Unique-Path Networks Stability Results}
\label{sec:single-path}

For the networks having tree structure (which is a special case of unique-path networks) the maximum stability region is achieved
under any work-conserving discipline, i.e. such that at any link communication epoch a useful packet is sent, unless none is available.
(Oldest-Useful, Random-Useful and Selfish are work-conserving disciplines.)

\begin{theorem}
If a network has tree structure, maximum stability region is achieved under any work-conserving discipline.
\end{theorem}

\begin{proof}
	We give only a sketch of the argument as the details are straightforward. Recall that the condition \eqn{eq-cut-cond} is necessary for stability.
	So, it suffices to show that, under any given work-conserving discipline, 
	the network is stable when condition \eqn{eq-cut-cond} holds. Again, we use the fluid limit technique (see Section ~\ref{sec:oldest-first}). 
	For any fluid limit consider the following.
	For a fixed link $(u,v) \in E$ let $V_{uv}$ denote the subset of those nodes which would be in the same graph component as $u$, if links $(u,v)$
	and $(v,u)$ would be removed. (By the definition of tree structure, it would be exactly two components.) Denote by $y_{uv}(t)$ the ``amount of fluid''
	that arrived at nodes $w\in V_{uv}$, by has not yet ``crossed'' link $(u,v)$. Then, at any regular time point $t$, such that $y_{uv}(t)>0$
	for at least one link $(u,v)$, we have
	\beql{eq-tree-deriv}
	\frac{d}{dt} \max_{(u,v)} y_{uv}(t)/|V_{uv}| < 0.
	\eeql
	This is because for any link $(u,v)$ for which the maximum in \eqn{eq-tree-deriv} is attained, there must be a non-zero amount of fluid
	included in $y_{uv}(t)$, which is already present at node $u$; then this fluid crosses link $(u,v)$ at the rate $c_{uv}$ which, by \eqn{eq-cut-cond},
	is strictly greater than the rate $\sum_{w\in K_{uv}} \lambda_w$ at which this fluid arrives into the network.
We omit further details.
\end{proof}

For general unique-path networks, we have the following 
\begin{theorem}
A unique-path network under Oldest-Useful discipline has maximum stability region.
\end{theorem}

\begin{proof} In a unique-path network, each packet originating at a node $u$ is disseminated along a fixed tree rooted at $u$.
(This is unlike the situation with symmetric networks -- see the comments in Section \ref{subsubsec:proof-highlights}.)
Given that, the stability proof in \cite{stolyar2003control} for a Largest-Weighted-Delay-First discipline applies essentially as is.
In fact, the stronger  -- large deviations -- results in \cite{stolyar2003control} apply as well, so that  Oldest-Useful is optimal 
in the sense of maximizing the decay rate of the tail of stationary distribution of the AoI in the network.
We refer reader to \cite{stolyar2003control} for details.
\end{proof}

\section{Relative Performance of Disciplines in Symmetric Systems: Simulation Results}
\label{sec:comparisons}

We compare the Oldest-Useful, Random-Useful, and Selfish disciplines via simulation. The version of Selfish discipline is such that,
when a node does not have own (originating at it) useful packets, it sends other useful packets also in the oldest-first order.

We simulate symmetric systems with the number of nodes $N = 50, 100, 150, 200$, link capacities $1$, and
with load values 
$ \lambda = 0.3, 0.5, 0.7$. (Recall that we proved that, under all three disciplines, $\lambda<1$ is necessary and
sufficient for stability.)
The actual simulation data is given in 
Table~\ref{tab:sim-data} in Appendix~\ref{sec-supporting-simulation}. It contains
the steady-state averages for the following performance metrics: the number of {\em distinct packets}, 
the number of {\em undelivered packet copies} and the AoI. They are defined as follows.
The number of distinct packets is the total number of packets currently present in the network.
(Note that, by Little's law, the mean number of distinct packets is equal to the mean network sojourn time times $\lambda N$.) 
If a packet is present at exactly $i \in \{1,2,\ldots, N-1\}$ nodes (i.e. it is a stage $i$ packet), 
we say that it has $N-i$ undelivered copies;
the number of undelivered packet copies is the total number of undelivered copies of all packets 
currently present  in the network.
Recall that AoI is the age of the oldest packet in the network.

We see that, as desired, Oldest-Useful discipline out-performs the Random-Useful discipline in terms of average AoI.
This cannot be the case for arbitrary networks, because, as we have shown, for some network topologies Oldest-Useful 
may not be even maximally stable; however, this is the case for symmetric networks.

The Random-Useful discipline out-performs the Oldest-Useful in terms of the average number of undelivered packet copies.
The Random-Useful discipline has a greater mean number of packets (which is proportional to the mean sojourn time) for $\lambda = 0.3$, but a lesser mean number of packets for $\lambda = 0.5, 0.7$.

Surprisingly, the Selfish discipline seems to have the best performance for the mean number of distinct packets, and especially for the AoI.

Regarding the dependence of the performance metrics on the network size $N$, as it becomes large,
this is the subject of Section~\ref{sec:delays}. There we will show that it is natural to expect that the mean sojourn time should grow as $O(\log N / N)$, and therefore (by Little's law) the mean number of distinct packets should grow as 
$O(\log N)$, and the mean AoI should grow as $O(\log N / N)$; finally, the mean number of undelivered packet copies 
should grow as $O(N \log N)$. At this point we note that the results in Table~\ref{tab:sim-data} (in Appendix~\ref{sec-supporting-simulation}) do appear to comply with these expectations, for each discipline.

\section{Observations and Conjectures about Packet Delays in 
a Large-Scale Symmetric System}
\label{sec:delays}
In this section, we consider 
the cumulative delays experienced by a packet as it propagates through the symmetric system,
specifically their limiting behavior as $N \to \infty$.
We state some conjectures about this behavior, supported by heuristic arguments and 
simulation experiments. Parts of the development in this section are rigorous -- the corresponding facts are stated as theorems, lemmas, propositions. The conjectures are stated as such. 

Throughout the section we consider a symmetric system,
with $N$ nodes, Poisson arrivals at each node at rate $\lambda < 1$, 
and each link capacity (communication rate) equal to $1$. 
Let us refer to this model as the {\em finite-capacity system,}
to distinguish it from a modified version, called free system, which we introduce next.

\subsection{Free System}
\label{sec-free-system}

Consider the following modification of the finite-capacity system: at any communication epoch at a link, {\em all} useful packets at this link are transmitted.
We call this the \emph{free system}, because here the propagation of each packet depends only on the communication epochs on the links; different packets do not ``interfere'' with each other in any way.
(One can also say that here each link has infinite transmission capacity.)
Therefore, from the perspective of any fixed packet, the process of its propagation through the network in the free system is equivalent to the well-known gossip model~\cite{ganeshnotes, shah2009gossip}. We will use the free system as a benchmark, with respect to which we will consider packet delays in the finite-capacity system. 

Sometimes, for notational convenience, we will refer to the free system as the system operating under the ``discipline'' labeled `free.' (This somewhat abuses the meaning of the term discipline, because here we actually have a different model.)
The system under any other discipline (RU, OU, ...) except `free,' is automatically a finite-capacity system.

Denote by $T^N_\mathrm{free}$ the (random) sojourn time of a packet in the free system with $N$ nodes.
\begin{lemma}
\label{lem-free-slowdown}
	The following holds:
\beql{eq-free1}
	\E[T^N_\mathrm{free}] = \sum_{i=1}^{N-1}\frac{1}{i(N-i)} = \frac{1}{N}\sum_{i=1}^{N-1}\left(\frac{1}{i} + \frac{1}{N-i}\right),
\eeql
\beql{eq-free2}
	\lim_{N\to\infty} \E[T^N_\mathrm{free}] /  [2 \log N /N] = 1,
\eeql
\beql{eq-free3}
	T^N_\mathrm{free} /  [2 \log N /N] \Rightarrow 1, ~~N \to \infty.
\eeql
\end{lemma}

\begin{proof}
	A  stage $i$ packet is useful on $i(N-i)$ links. Recall that communication epochs on each link follow a unit rate Poisson process.
	Thus, the expected time for a packet to move from stage $i$ to stage $i+1$ is $\frac{1}{i(N-i)}$.
	Then \eqn{eq-free1} follows, since the packet departs the system immediately upon reaching stage $N$.
	Relation \eqn{eq-free2} obviously follows from \eqn{eq-free1}; \eqn{eq-free3} follows by showing that the variance 
	of left-hand side goes to $0$.
\end{proof}

\subsection{Normalized delays}

\subsubsection{Normalized sojourn time}

For a fixed discipline and a fixed $N$, the {\em normalized (steady-state) sojourn time}  is the sojourn time normalized (divided by) $\E T^N_\mathrm{free}$; the mean normalized sojourn time will be called the \emph{slowdown}. 

The simulation results in Figure~\ref{fig:OU-NST} 
and Table~\ref{tab:slowdown} (in Appendix~\ref{sec-supporting-simulation}), suggest 
the following
\begin{conj}
	Under RU and OU, for a given 
$\lambda$, as $N\to\infty$, the distribution of the normalized sojourn time converges 
to a fixed distribution (depending on the discipline); 
and the slowdown converges to the expected value of that distribution. 
	\label{conj:slowdown-conv}
\end{conj}

\begin{figure}[htp!]
	\centering 
    \begin{subfigure}[t]{0.43\textwidth}
        \centering
        \includegraphics[scale=0.48]{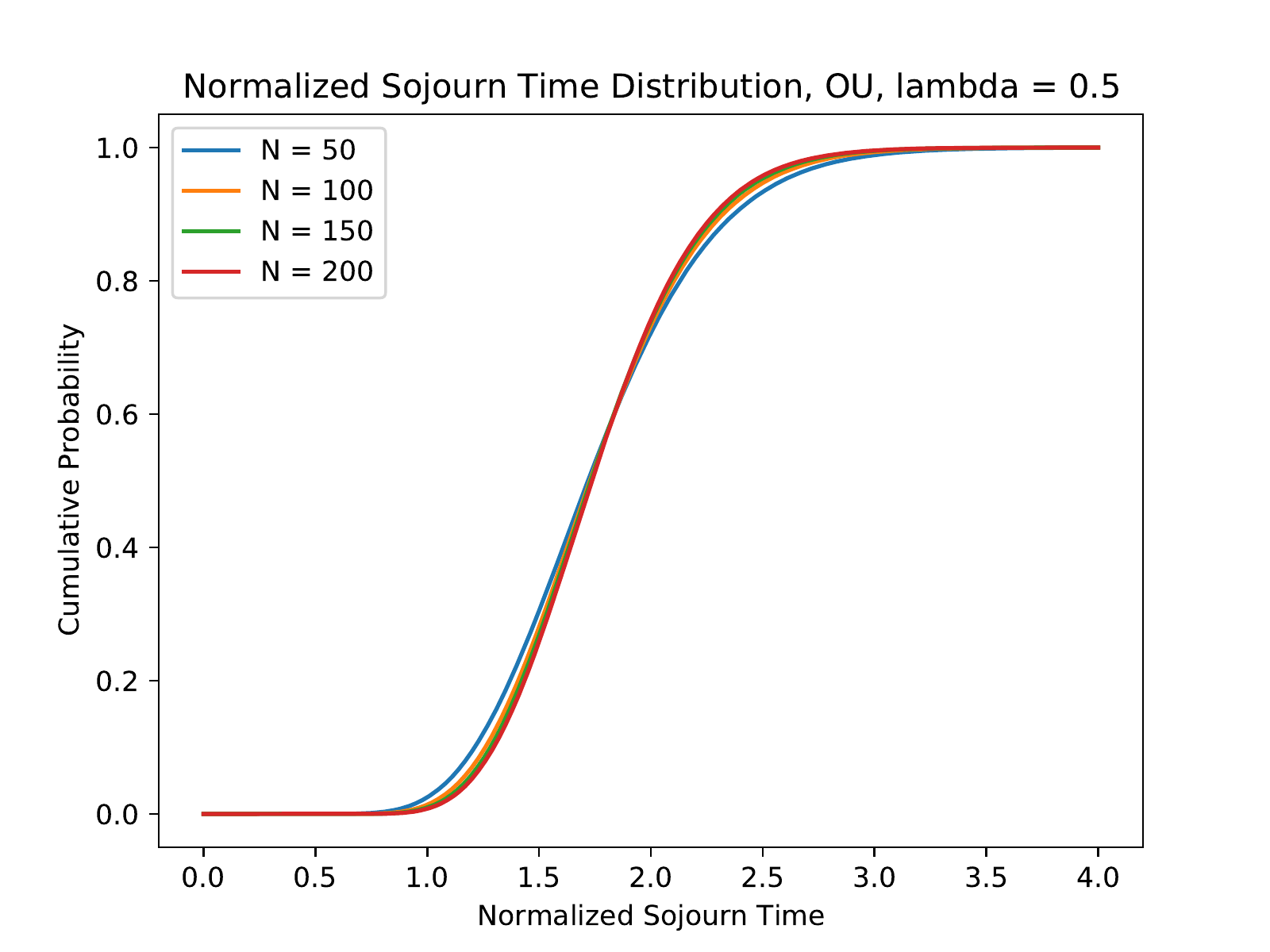}
    \end{subfigure}
    ~ 
    \begin{subfigure}[t]{0.43\textwidth}
        \centering
        \includegraphics[scale=0.48]{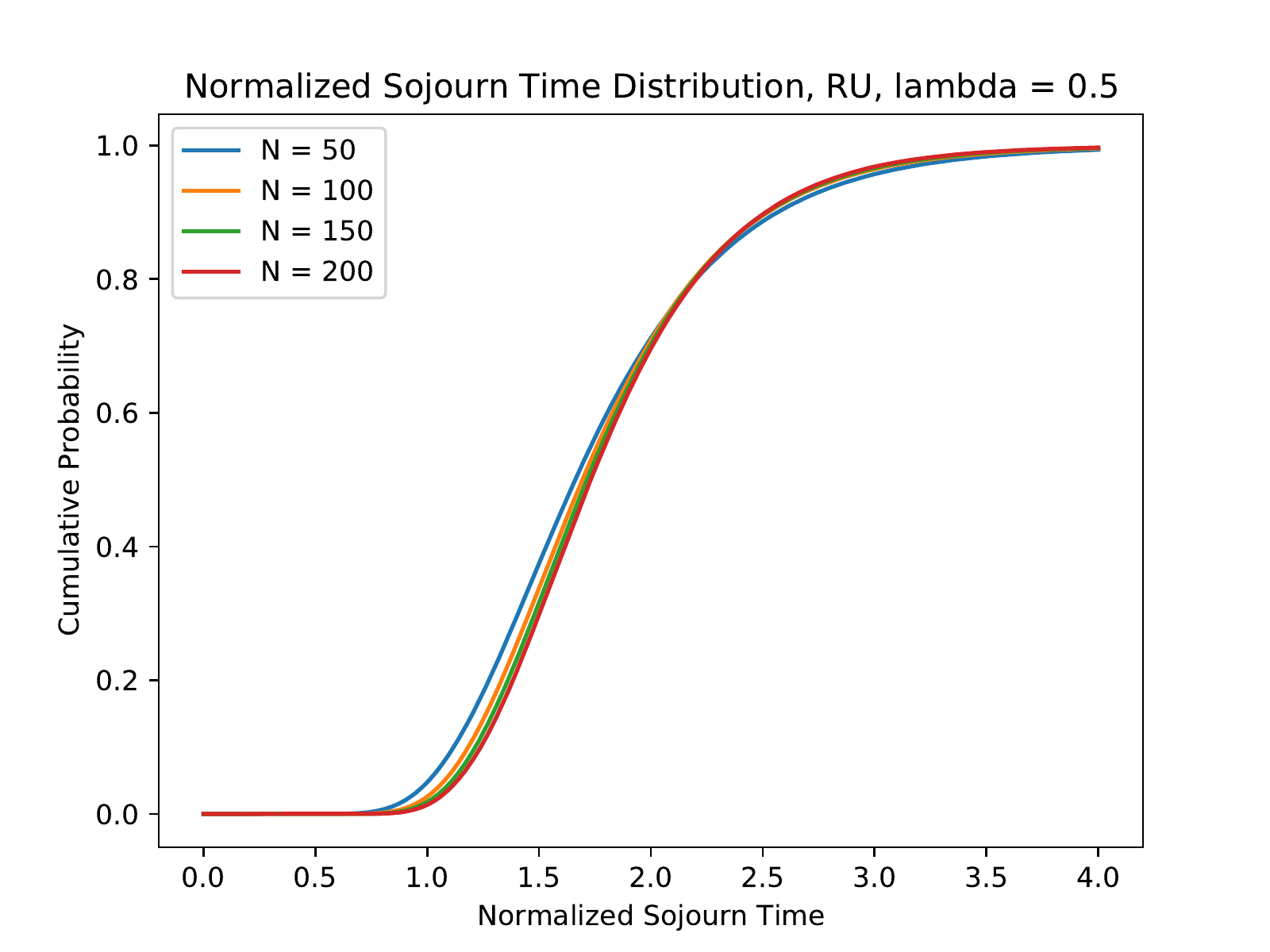}
    \end{subfigure}

    \caption{Simulated normalized sojourn time distributions for the OU and RU disciplines, load 0.5. (The results for loads 0.3 and 0.7 also show very little dependence on $N$.)}
    \label{fig:OU-NST}
\end{figure}

\subsubsection{Natural Space-Time Transformation}

We will also consider the expected times for a packet to reach different stages. For that it is convenient to consider the process of a packet propagation through the stages under the following space and time transformation. 
We will use notation:
$$
B_N := \E[T^N_\mathrm{free}],
$$
 $$
 A_N := N\E[T^N_\mathrm{free}] = \sum_{i=1}^{N-1}\frac{N}{i(N-i)} = \sum_{i=1}^{N-1}\left(\frac{1}{i} + \frac{1}{N-i}\right).
 $$
Consider a system with given $N$. First, we speed up time by the factor $B_N$, so that in a free system the expected sojourn time of a packet is exactly $1$. Second, with each stage $i=1,2, \ldots, N-1, N$, we associate a point $\gamma_{i-1}^N$ within the 
continuous segment $[0,1]$:
$$
\gamma_0^N := 0, ~~~~~\gamma_i^N := \frac{1}{B_N}\sum_{k=1}^{i}\frac{1}{k(N-k)}, ~~i = 1, 2, \ldots, N-1.
$$
With a packet at stage $i=1, \ldots, N-1$ in the original system we associate a particle located at point $\gamma_{i-1}^N$ in the transformed system. (So, in transformed system packets arrive at point $0$ and leave when 
they reach point $1=\gamma_{N-1}^N$, which corresponds to reaching stage $N$ in the original system.) 
Note that the distance from $\gamma_{i-1}^N$ to $\gamma_{i}^N$,
\beql{eq-int-length}
\gamma_{i}^N - \gamma_{i-1}^N = \frac{1}{B_N} \frac{1}{i(N-i)}
\eeql
is exactly equal to the expected time for a particle to move from $\gamma_{i-1}^N$ to $\gamma_{i}^N$ in the free transformed system. 

\begin{rem}
\label{rem-free-slowdown}
Given the definition of the transformed system, it is easy to see that, in the $N\to\infty$ limit, the movement of any particle in the transformed free system is such that the particle arrives at point $0$, moves right deterministically at the constant speed $1$, and departs upon reaching point $1$.  
\end{rem}

In what follows, the term {\em original system} is used to distinguish it from the transformed system.

For future reference, we need the following fact.

\begin{proposition} 
	For any $\epsilon >0$, as $N\to\infty$, the fraction of the points $\{\gamma_k^N, ~k=1,\ldots,N-1\}$
	located in $(1/2-\epsilon,1/2+\epsilon)$, converges to $1$.
	\label{prop:pi-to-half-1}
\end{proposition}

\begin{proof}
The claimed property is equivalent to the following one. For any fixed $\alpha \in (0, 1)$ and any integer-valued sequence 
$k=k(N)$ such that $k/N \to \alpha$, $\lim_{N \to \infty} \gamma_k^N = 1/2$. Let us prove this.

Recall that
$$
\gamma_k^N = \frac{1}{B_N N}\sum_{i=1}^{k}\left( \frac{1}{i} + \frac{1}{N-i} \right),
$$
$\lim_{\ell \to \infty} \log \ell /[\sum_{i=1}^{\ell-1}\frac{1}{i}] = 1$, and $\lim_{N \to \infty} N B_N /  (2 \log N) = 1$.
Then,
	\begin{align*}
		\lim_{N\to\infty} \gamma_k^N & = \lim_{N \to \infty}\frac{\log k + \log N - \log(N - k)}{2 \log N}	 \\
		& = \lim_{N \to \infty}\frac{\log N + \log\frac{k}{N - k}}{2 \log N} = \frac{1}{2}.
	\end{align*}
\end{proof}

\subsubsection{Normalized Delay Profile}

For a fixed discipline D (which may be OU, RU, free, or some other discipline), 
denote by $T^N_D(x)$, the time (or, delay) for a particle to reach or cross point $x \in [0,1]$ in the 
(steady-state) transformed system. Then, $R^N_D(x)=\E T^N_D(x)$ we will call the \emph{mean normalized delay to point $x$},
and the function $R^N_D(x)=\E T^N_D(x), x \in [0,1]$, 
will be called the  \emph{mean normalized delay profile}. Clearly, $R^N_D(1)$ is nothing else but the 
mean normalized sojourn time, i.e. the slowdown. 

Generalizing the argument used in the proof of Lemma~\ref{lem-free-slowdown} (see also Remark~\ref{rem-free-slowdown}), it is easy to obtain the following fact about free system.

\begin{lemma} 
\label{lem-free-delay}
As $N\to\infty$, for each $x\in [0,1]$,
$$
R^N_\mathrm{free}(x) \to x,
$$
$$
T^N_\mathrm{free}(x) \Rightarrow x.
$$
\end{lemma}

Simulated mean normalized delay profiles for the OU and RU disciplines are shown in Figure~\ref{fig:OU-MNDF}.
 Based on these results, we make the following
\begin{conj}
	For the OU and RU disciplines, i.e. D = OU, RU, for a given $\lambda<1$, as $N\to\infty$, 
	the mean normalized delay profile converges,
	$$
	R^N_D(x) \to R_D(x), ~~0 \le x\le 1,
	$$
	where $R_D(\cdot)$ is a continuous increasing function.
	\label{conj:mndf-conv}
\end{conj}

\begin{figure}[htp!]
	\centering 
    \begin{subfigure}[t]{0.43\textwidth}
        \centering
        \includegraphics[scale=0.48]{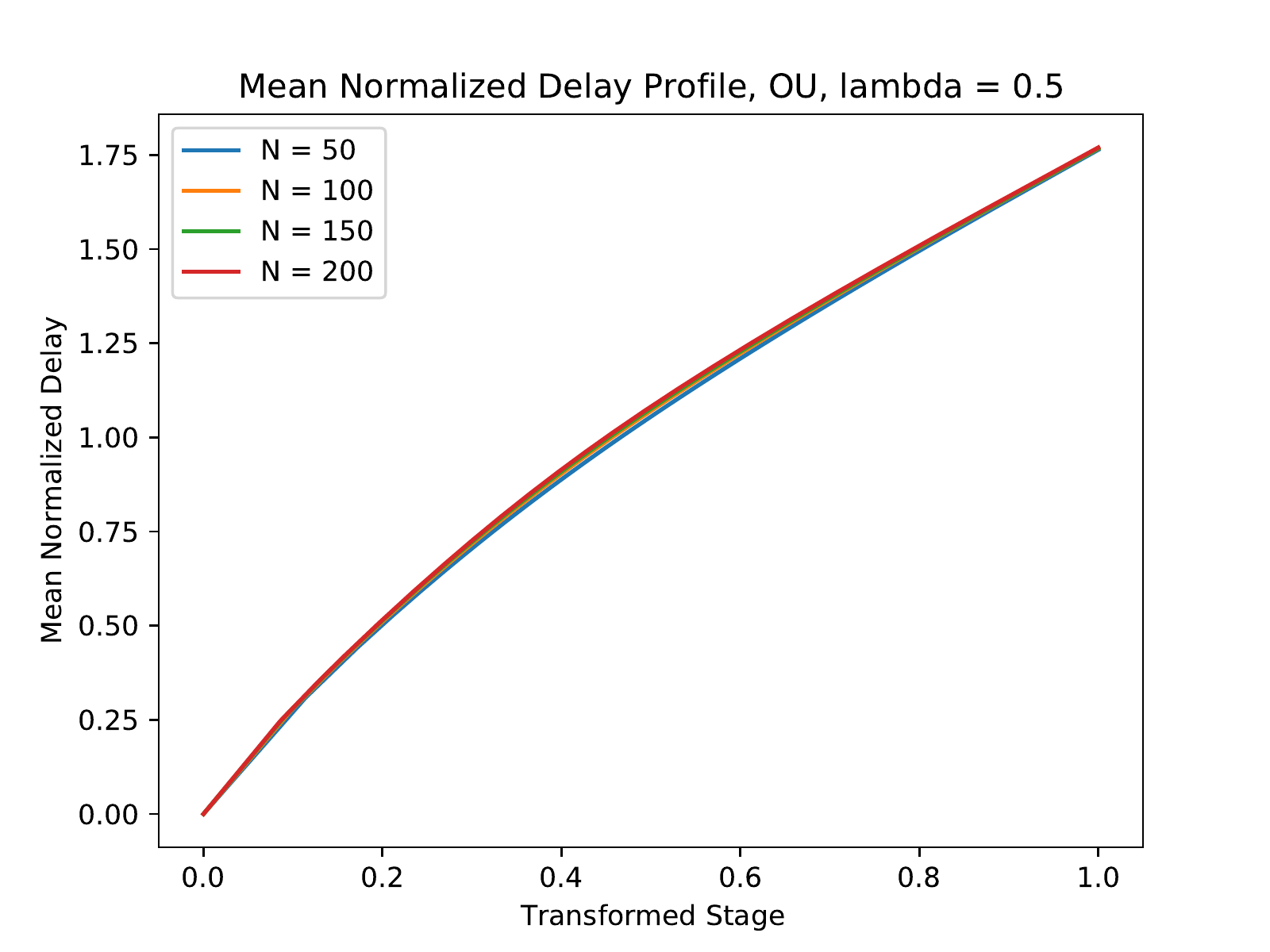}
    \end{subfigure}
    ~
        \begin{subfigure}[t]{0.43\textwidth}
        \centering
        \includegraphics[scale=0.48]{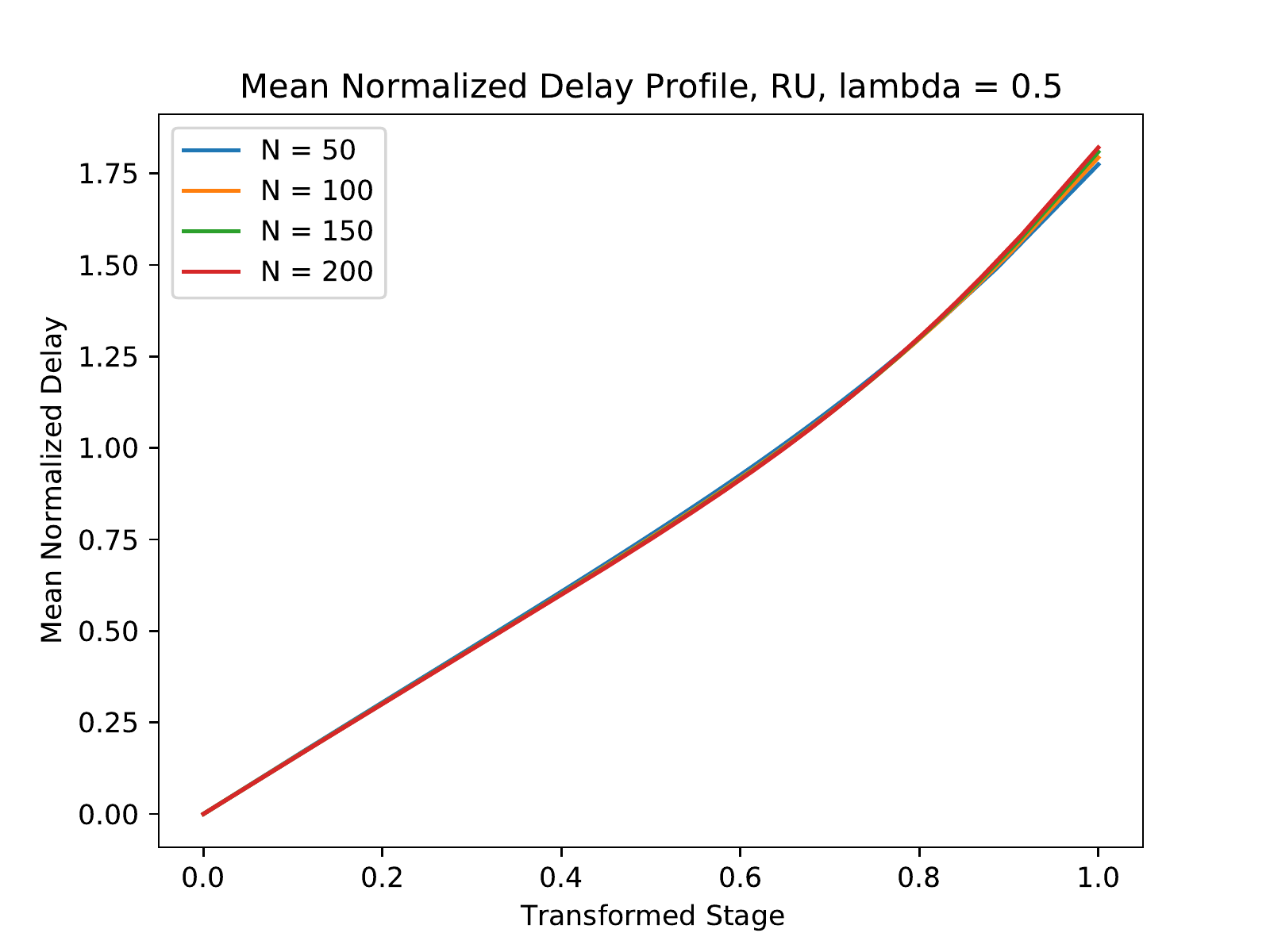}
    \end{subfigure}

    \caption{Mean normalized delay profile for the OU and RU disciplines, load 0.5. \\ (The results for loads 0.3 and 0.7 also show very little dependence on $N$.)}
    \label{fig:OU-MNDF}
\end{figure}

\section{System with ``Reshuffling''}
\label{sec:reshuffling}

To obtain some analytic insight into Conjectures~\ref{conj:slowdown-conv} and~\ref{conj:mndf-conv}, we consider the following artificial system, with ``reshuffling,'' which is more tractable than the system under the RU and OU disciplines.
In particular, for this model, we give heuristic arguments which lead to explicit expressions for the limiting slowdown and the limiting mean normalized delay profile, as $N\to\infty$.
These arguments and expressions are well-confirmed by simulation results. Just as in Section~\ref{sec:delays},
parts of the development are rigorous. Those parts are stated as as theorems, lemmas, propositions. The conjectures are stated as such. 

\subsection{Reshuffling Model}

Consider the symmetric finite-capacity system, specified at the beginning of Section~\ref{sec:delays}, with $\lambda<1$. 
Consider the following modification of this model, which we call the {\em reshuffling system}. 
Immediately before
any communication epoch, each packet in the network is removed form all the nodes it occupies, say $i$ nodes, and placed back into a subset of $i$ nodes chosen uniformly at random.
Then, upon a communication epoch on a link, a useful packet (if any), chosen uniformly at random, is transmitted (just as under RU discipline). Again, with some abuse of notion of a discipline, we sometimes refer to this a system operating under the reshuffling discipline, labeled RS.

Clearly, the evolution of the reshuffling network is described by a countable irreducible Markov chain
$Y(t)=(n(i, t), ~i = 1, \ldots, N-1)$, where $n(i, t)$ is the number of stage $i$ packets  
at time $t$.

\subsubsection{Stability of Reshuffling System}
\begin{theorem}
	For any $N$, the process describing the reshuffling system is stable (positive recurrent)  if and only if $\lambda < 1$.
\end{theorem}
\begin{proof}
The necessity of $\lambda < 1$ is straightforward. The sufficiency of $\lambda < 1$ for stability is also 
easy to obtain, using the Lyapunov-Foster stability criterion (cf. \cite{foss2004overview,Bramson-book}) with Lyapunov function being 
the total number of undelivered packet copies in the system. (Recall that a packet at stage $i$ has $N-i$ undelivered copies.)
Indeed, if the number of undelivered copies is large, then the number of packets in the system is also large;
then, the probability that a link at a transmission epoch actually transmits a packet is very close to $1$;
therefore, the average rate at which the number of undelivered copies is decreasing due to transmissions
is very close to $N(N-1)$, while that rate at which it increases due to new new arrivals is $\lambda N(N-1)$.
We omit straightforward details.
\end{proof}

\subsection{Normalized Delay Profile in the Large Reshuffling System} 
\label{sec-delay-profile}

We make the following conjecture about the behavior of a large reshuffling system in steady-state.
\begin{conj}
\label{conj:reshuf-Poisson}
Fix $\lambda<1$. For each $N$, consider the reshuffling system in steady-state. As $N\to\infty$:\\
(i) The distribution of the number of useful packets on a link at a communication epoch, converges to 
$\mathrm{Poisson}(\psi)$,
where $\psi \ge 0$ is a fixed number.\\
(ii) The numbers of useful packets at the links at consecutive communications epochs are asymptotically independent.\\
(iii) In fact, $\psi = -\log(1-\lambda)$. \\
(iv) The slowdown converges to 
$-\log(1-\lambda)/\lambda$, and moreover, for each $x\in [0,1]$,
$$
R^N_\mathrm{RS}(x) \to [-\log(1-\lambda)/\lambda] x,
$$
$$
T^N_\mathrm{RS}(x) \Rightarrow [-\log(1-\lambda)/\lambda] x.
$$
(Compare to Lemma~\ref{lem-free-delay} for the free system.)
	\end{conj}
	
The motivation and simulation evidence for this conjecture are as follows. 	

Based on the free system, which in terms of the total number of packets is a lower bound of the reshuffling (or any other) system, the typical number of packets in the system with large $N$ is large, of the order $\log N$. If we consider a communication epoch at a link, then, due to reshuffling, the events of different packets being useful at (``competing for'') this link are independent, with a packet at stage $i$ becoming useful with a ``typically'' small probability $i(N-i) / [N(N-1)]$. Thus, the number of useful packets is a sum of a large number of independent binary random variables, each having value $1$ with a small probability. This naturally leads to Conjecture~\ref{conj:reshuf-Poisson}(i), which is also supported by simulation experiments (see Figure~\ref{fig:useful-packets-distn} in Appendix~\ref{sec-supporting-simulation}). 

The Conjecture~\ref{conj:reshuf-Poisson}(ii) is very natural due to the argument just above and the fact that 
the distribution of the number of useful packets at a communication epoch depends only on the process state 
$(n(i, t), ~i = 1, \ldots, N-1)$, which ``does not change much'' over a finite sequence of consecutive epochs.

Now, the value $\psi = -\log(1-\lambda)$ (Conjecture~\ref{conj:reshuf-Poisson}(iii)) is obtained as follows. 
The probability that any link at its communication epoch will actually transmit a packet is $1-e^{-\psi}$; but this has to be equal $\lambda$, which is the steady-state utilization. Therefore, $\lambda=1-e^{-\psi},$ which can be solved for $\psi$.
Table~\ref{tab:poisson-parameters} provides simulation evidence that, when $N$ is large, 
$\psi$ is indeed close to $-\log(1-\lambda)$, as well as additional evidence that the distribution of the number 
of useful packets is indeed close to Poisson (Conjecture~\ref{conj:reshuf-Poisson}(i)).

To justify Conjecture~\ref{conj:reshuf-Poisson}(iv), consider 
a packet, which happens to be useful at a link at its communication epoch. Since this packet is useful with only a small probability, independently of other packets, we conclude that the number of {\em other} useful packets at this link
still has the distribution $\mathrm{Poisson}(\psi)$. If so, 
the probability that our fixed packet actually gets transmitted is
$\E 1/(1+A),$
where $A$ has distribution $\mathrm{Poisson}(\psi)$. This expectation is 
\beql{eq-speed-density}
v := \E \frac{1}{1+A}=\frac{1-e^{-\psi}}{\psi}.
\eeql
Note that in the free system a packet which is useful on a link is certainly transmitted at its communication epoch.
We conclude that in the reshuffling system in steady-state a packet propagates $1/v= \psi/[1-e^{-\psi}]= -\log(1-\lambda)/\lambda$ times ``slower'' than in the free system, which is the limiting slowdown (as $N$ becomes large). 
Finally, the fact that, as $N\to\infty$, the mean normalized delay profile becomes linear,  
$[-\log(1-\lambda)/\lambda]x, ~0\le x \le 1$, is natural to expect, because the random number of other packets,
competing with a given packet on a link, should not depend on the stage of the packet. This completes the informal 
motivation for Conjecture~\ref{conj:reshuf-Poisson}(iv). Figure~\ref{fig:RS-MNDF} provides strong simulation evidence 
for it. (Conjecture~\ref{conj-rs-dynamics} in Appendix~\ref{sec-profile-dynamics}  gives an additional intuition for Conjecture~\ref{conj:reshuf-Poisson}(iv), based on the system dynamics over time.)

We note that the expression for the (limiting) slowdown, which we obtained above, can be written as
 $1/v=\psi/\lambda$ (where we used relation $\lambda=1-e^{-\psi}$). This expression can also be obtained as follows.
 In the free system, the expected number of useful packets on a link at its communication epoch is exactly $\lambda$
 (because this is the average number of packets transmitted by a link at its communication epoch, in steady-state). In the reshuffling system the packets move ``slower,'' which results in the larger number $\psi$ of useful packets at a link in steady-state; hence the (limiting) slowdown must be the ratio $\psi/\lambda$.
 
 \begin{rem}
 \label{poisson-conj-remark}
 Even though Conjecture~\ref{conj:reshuf-Poisson} is very intuitive and well supported by simulations, proving it formally poses significant challenges. (Same is true for a related Conjecture~\ref{conj-reshuf-location} below.) For example, the intuition leading to 
 the key Conjecture~\ref{conj:reshuf-Poisson}(i) is based on the assumption that the
 probability $i(N-i) / [N(N-1)]$ of a packet at stage $i$ becoming useful 
 on a link at its communication epoch is ``typically small.'' This is not true for a packet at any stage; say, if stage $i$ is such that $i-N/2=O(1)$, this probability is close to $1/4$. This is ``mitigated'' by the fact that the total number of packets in the network ``should be'' $O(\log N)$, i.e. much smaller than $N$, and therefore at a given communication epoch the probability of having packets at stages $i$ to be ``close to'' $N/2$ is likely to be small. (Also,
 packets at stages close to $N/2$ propagate along the stages faster.) 
 However, making this kind of argument formal is challenging. Proving Conjectures~\ref{conj:reshuf-Poisson} 
 and \ref{conj-reshuf-location} may be a subject of future work.
 \end{rem}

\begin{table}
\centering	
\begin{tabular}[]{| c || c | c | c | c | c |}
	\hline
	 & $-\log(1 - \lambda)$ & MLE, $N = 50$ & MLE, $N = 100$ & MLE, $N = 150$ & MLE, $N = 200$ \\
	\hline
	$\lambda = 0.3$ & $0.375$ & $0.381$ & $0.372$ & $0.368$ & $0.369$ \\
	$\lambda = 0.5$ & $0.693$ & $0.715$ & $0.706$ & $0.706$ & $0.701$ \\
	$\lambda = 0.7$ & $1.204$ & $1.253$ & $1.245$ & $1.240$ & $1.238$ \\
	\hline
\end{tabular}
\caption{The predicted mean $\psi=-\log(1-\lambda)$ for the number of useful packets on a link, according to Conjecture~\ref{conj:reshuf-Poisson}(iii), and the maximum likelihood estimate (MLE) for the same mean based on simulation
and Conjecture~\ref{conj:reshuf-Poisson}(i).} 
\label{tab:poisson-parameters}
\end{table}

\begin{figure}
\centering
        \includegraphics[scale=0.48]{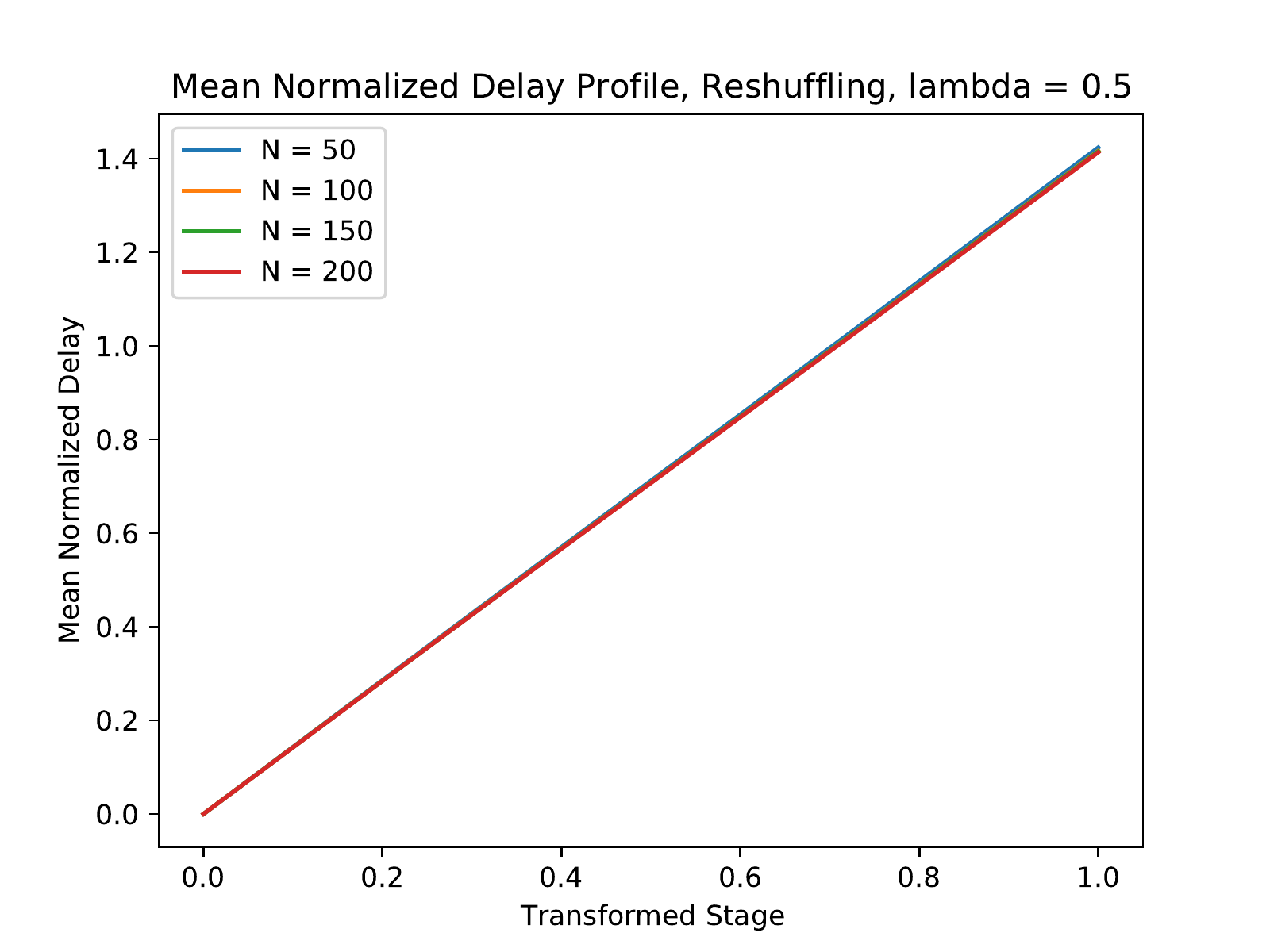}
    \caption{Simulated mean normalized delay profile for the Reshuffling system, load 0.5. (The results for loads 0.3 and 0.7 also show essentially a linear function, with very little dependence on $N$.)}
    \label{fig:RS-MNDF}
\end{figure}

\subsection{Particle stage profile in the large reshuffling system}
\label{sec-particle-distr} 

We now consider the question of the distribution of the particle stages in the system. 
Notice that by Little's law, $(\lambda N) B_N R^N_\mathrm{free}(1) = \lambda A_N$ 
is the expected number of particles in the free system in steady-state. When we consider the numbers 
of particles a transformed system, 
it will be convenient to rescale (divide) them by $A_N$, so that each particle has ``weight'' $1/A_N$ and the total  ``mass'' of all particles in the system is of order $1$. Correspondingly, let us denote by $\phi^N_D(x,t)$, $x \in [0,1]$, 
$t \ge 0$, the quantity of particles in the transformed system, located in $[0,x]$ at time $t$, under discipline $D$. (In particular, $D$ = free and $D$ = RS, for the free and reshuffling disciplines, respectively.) We call $\phi^N_D(\cdot,t)$ the stage profile at time $t$,
and denote by $\phi^N_D(\cdot,\infty)$ the steady-state stage profile. Note that,
due to the time and particle quantity rescaling, the rate at which the particle mass arrives into a transformed system is $(\lambda N) (1/A_N) B_N = \lambda$.

Then, for the free system, we have the following property, which is essentially an alternative form of Lemma~\ref{lem-free-delay}.

\begin{lemma} As $N\to\infty$, 
$$
\sup_{0 \le x\le 1} |\phi^N_{free}(x,\infty) - \lambda x| \Rightarrow 0.
$$
\end{lemma}

This lemma is intuitive, because the new particle mass arrives at rate $\lambda$, and the particles move independently, essentially at the constant rate $1$ when $N$ is large. The formal proof is also straightforward.

Recall that Conjecture~\ref{conj:reshuf-Poisson}(iv) basically states that particles in a large system with reshuffling,
in steady-state, 
move $-\log(1-\lambda)/\lambda$ times slower than in the free system. Then,
our conjecture about the stage profile, which corresponds to Conjecture~\ref{conj:reshuf-Poisson}(iv), is the following

\begin{conj}
\label{conj-reshuf-location}
Let $\lambda<1$ be fixed. For each $N$, consider the reshuffling system in steady-state. As $N\to\infty$, $\phi^N_{RS}(\cdot,\infty)$ converges (in probability) to the deterministic linear function
$\psi x, ~0\le x \le 1$, where $\psi=-\log(1-\lambda)$.
	\end{conj}

If this conjecture is correct, then the parameter $\psi=-\log(1-\lambda)$ of the Poisson distribution, which we found earlier
from the relation $\lambda=1-e^{-\psi},$ is equal the constant density of the particle mass in the transformed 
system with reshuffling in steady-state, when $N$ is large. Then, \eqn{eq-speed-density} can be viewed as the mapping taking the (constant) density $\psi$ of the particle mass into the speed $v$ at which the particles move.

To further substantiate Conjecture~\ref{conj-reshuf-location}, one can study the dynamics of the stage profile 
$\phi^N_{RS}(\cdot,t)$ over time $t$. This analysis is given in Appendix~\ref{sec-profile-dynamics}, in which we
provide a heuristic argument leading to Conjecture~\ref{conj-rs-dynamics} (Appendix~\ref{sec-profile-dynamics}), 
stating that, when $N$ is large, this dynamics becomes deterministic, and we specify this dynamics. Furthermore, we rigorously show (in Theorem~\ref{th-hydro} in Appendix~\ref{sec-profile-dynamics}) that the specified dynamics  is such that 
the stage profile converges to the linear function, given in Conjecture~\ref{conj-reshuf-location}. In Appendix~\ref{sec-profile-dynamics} we also describe a simulation experiment, supporting 
the validity of Conjecture~\ref{conj-rs-dynamics} (Appendix~\ref{sec-profile-dynamics}).

\section{Conclusions}
\label{sec-conclusions}

For the model of data flows' dissemination over a network, we study the questions of stability and packet propagation delays,
under several communication disciplines and different network structures.

Our stability results prove, in particular, that while the Random-Useful discipline always provides maximum stability, 
a very natural Oldest-Useful discipline does not. For an important special case of the symmetric system 
we prove that Oldest-Useful does achieve maximum capacity.

For the issue of packet propagation delays, we put forward several intriguing conjectures, and support them by heuristic arguments and simulations. Most importantly, we conjecture that, in the symmetric network, 
as its size $N\to\infty$, the mean normalized delay profile under
the Oldest-Useful or Random-Useful discipline converges to a fixed continuous function (depending on the discipline).
In particular, this means that the ratio of the mean steady-state packet sojourn time to that in the free system (where 
a packet has the entire network just for itself) remains bounded as $N\to\infty$. 
More formal analysis of the mean normalized delay profile under Oldest-Useful, Random-Useful and possibly other disciplines,  is an interesting subject for future research.

	\bibliographystyle{plain}

\appendix

\section{Proof of Theorem~\ref{th-ru-generalization}}
\label{sec-ru-gen-proof}

The proof is a relatively minor adjustment of the proof of theorem 3 in \cite{massoulie2008rate}. 
For this reason we do not give a self-contained proof here, but rather specify the required changes
in that proof. This means that an interested reader will have to read the proof of theorem 3 in \cite{massoulie2008rate},
and then ``apply'' the adjustments that we now specify.

The proof of theorem 3 in \cite{massoulie2008rate} uses the fluid limit technique, which is also discussed and used in our  Section \ref{sec:oldest-first}.
The following key estimate is derived in that proof: 
$$
\frac{\mathrm{d}}{\mathrm{d}t}y_{\subseteq S^*} \le \lambda_s - \left[ \sum_{\substack{u \in S^* \\ v \not\in S^*}}c_{uv} - (\max_{e \in E}c_e|E|2^K)\alpha \right].
$$
Here the left-hand side is the derivative of the ``amount of fluid" that originated in the subset of nodes $S^*$, containing the single source node $s$,
but has not ``left the subset $S^*$ yet,'' i.e. not present in any node outside $S^*$.
The right-hand side is equal to the rate at which such fluid arrives (equal to $\lambda_s$, since $s$ is the single source) minus 
the lower bound (the term in brackets) on the rate at which such fluid leaves subset $S^*$.
In the lower bound (the term in brackets): $\sum_{\substack{u \in S^* \\ v \not\in S^*}}c_{uv}$ is the capacity of the graph cut from 
$S^*$ to $V \setminus S^*$; $(\max_{e \in E}c_e|E|2^K)$ is a constant depending only on the model parameters,
and $\alpha>0$ can be chosen arbitrarily small. It can be checked 
that the same lower bound holds as is for a multi-source network as well. Consequently,  in our multi-source case, the 
above derivative estimate still holds if we replace the fluid arrival rate $\lambda_s$ in the right-hand side by $\sum_{u\in S^*} \lambda_u$:
$$
\frac{\mathrm{d}}{\mathrm{d}t}y_{\subseteq S^*} \le \sum_{u\in S^*} \lambda_u - \left[ \sum_{\substack{u \in S^* \\ v \not\in S^*}}c_{uv} - (\max_{e \in E}c_e|E|2^K)\alpha \right].
$$
Given this, the rest of the proof of theorem 3 in \cite{massoulie2008rate} applies, if we use \eqn{eq-cut-cond} (with $S = S^*$)
instead of $\lambda_s < \sum_{\substack{u \in S^* \\ v \not\in S^*}}c_{uv}$.

\section{Dynamics of the stage profile in the large reshuffling system}
\label{sec-profile-dynamics}

\subsection{Hydrodynamic model definition and convergence. Main conjecture.}

Suppose that $N$ is very large and a stage-profile $\phi^N_{RS}(\cdot)$
is close to a fixed linear stage-profile $\xi x, ~0\le x \le 1$, where $\xi >0$ is a constant density. Let us see what is the corresponding average number $H$ of useful packets, competing for a link at its communication epoch. Based on the discussion following Conjecture~\ref{conj-reshuf-location}, we should expect that $H=\xi$. Let us see if this is the case. Consider a pre-limit system with large $N$. Recall that the point $\gamma^N_{i-1}$ in the transformed system corresponds to stage $i=1,2,\ldots, N-1$
in the original system, and the length of $[\gamma^N_{i-1},\gamma^N_{i})$ interval is given in \eqn{eq-int-length}.
Then the expected number of stage $i$ packets in the original system is approximately
$$
\xi \frac{1}{B_N} \frac{1}{i(N-i)} A_N.
$$
Each stage $i$ packet competes for given link with probability $[i(N-i)]/[N(N-1)]$. Therefore, the expected number of stage $i$ packets competing for the link is
$$
\xi \frac{1}{B_N} \frac{1}{i(N-i)} A_N \frac{i(N-i)}{N(N-1)} =\xi \frac{1}{N-1}.
$$
We see that the contribution of each stage $i$ (in the original system) into $H$ is equal to $\xi/(N-1)$,
and therefore $H=\xi$, as we expected. Note that, given the density $\xi\ge 0$, the corresponding average speed
of any particle is given by the function
\beql{eq-speed-density2}
v(\xi)= \frac{1-e^{-\xi}}{\xi},
\eeql
which is just \eqn{eq-speed-density} with $\psi$ replaced by $\xi$. We adopt the convention that, in \eqn{eq-speed-density}, $v(0)=1$. 

Now, fix a small $\epsilon >0$ and denote by $H(\epsilon)$ the contribution into $H$ of those stages $i$ 
for which $\gamma^N_{i-1} \in (1/2-\epsilon,1/2 + \epsilon)$, i.e. with corresponding locations $\gamma^N_{i-1}$ in the transformed system lying in the small interval $(1/2-\epsilon,1/2 + \epsilon)$. 
Since the contributions of all stages are equal, namely $\xi/(N-1)$, by using Proposition~\ref{prop:pi-to-half-1} we conclude that,
as $N\to\infty$, $H(\epsilon)/H \to 1$. Then, when $N$ is large, the value $H$ is determined by 
by the stages located in the transformed system in a small neighborhood of the middle point $1/2$. In other words, 
``only the density at the middle point $1/2$ determines $H$.''

From here it is easy to conclude that, when $N$ is large, if we have a stage-profile $\phi^N_{RS}(\cdot)$
close to a deterministic stage-profile $\phi(\cdot)$ with any (not necessarily constant) bounded density 
$\xi(x) = \phi'(x), ~0\le x \le 1,$
then, when $N$ is large, the average number $H$ of useful packets, competing for a link at its communication epoch
is equal to $\xi(1/2)$. This, in turn, implies that the instantaneous speed of particles in the transformed system
is 
\beql{eq-speed-density3}
v(\xi(1/2))= \frac{1-e^{-\xi(1/2)}}{\xi(1/2)}.
\eeql

We are now in position to describe the dynamics of the deterministic stage-profile $\phi(\cdot,t)$,
approximating the dynamics of $\phi^N_{RS}(\cdot,t)$ in the reshuffling system when $N$ is large.
Let $\xi(x,t) = (\partial /\partial x) \phi(x,t)$ denote the density of $\phi(\cdot,t)$. 
The dynamics is such that the particle mass, at any point $x\in [0,1)$,
is shifted to the right at the instantaneous speed $v(\xi(1/2,t))$. Given that the new particle mass arrives at point $0$
at the constant rate $\lambda$, the density $\xi(0,t)$ that ``appears'' at the left boundary point $0$ must be $\lambda/v(\xi(1/2,t))$. 
Consider the mapping
$$
M(u) = \frac{\lambda}{v(u)} = \frac{\lambda u}{1-e^{-u}}, ~~u\ge 0,
$$
which takes $\xi(1/2,t)$ into $\xi(0,t)$. Recall that we consider $\lambda<1$. It is easily checked that: mapping $M(\cdot)$ is continuous increasing; it has unique fixed point $\psi=-\log(1-\lambda)$; a sequence of iterations $u_{n+1} = M(u_n)$, $n=0,1,\ldots$, is monotone strictly increasing [resp., strictly decreasing], converging to $\psi$ when $u_0 < \psi$ [resp., $u_n>\psi$].

Observe that the following relation holds for any $t$ and any $x\le 1/2$:
\beql{eq-one-half}
\xi(x, t) = M(\xi(x+1/2,t))= \frac{\lambda}{v(\xi(x+1/2,t))}.
\eeql
Given this special structure of $\xi(x,t)$, we see that, for any bounded (measurable) non-negative initial condition
$\xi(x,0), ~ 0 \le x \le 1,$ the unique solution is as follows. Let us (uniquely) extend $\xi(x,0)$ to all $x \in (-\infty, 1]$ via
\eqn{eq-one-half}. By the properties of the mapping $M$, this extension is well-defined, 
the function $\\sup_{x \le 1/2} \xi(x,0) \le \sup_{0\le x\le 1/2} \xi(x,0)$
and, moreover, 
\beql{eq-conv-on-the-left}
\xi(x,0) \to \psi=-\log(1-\lambda), ~~ x \to -\infty.
\eeql
Define function 
$$
\tau(y) = \int_{1/2-y}^{1/2} \frac{1}{v(\xi(z,0))} dz, ~~ y \ge 0.
$$
This is the time it takes for the total displacement (to the right) to reach value $y$; in other words, this is
the time it takes for the mass initially located at point $1/2 -y$ to reach point $1/2$. 
By the boundedness of $\xi(\cdot,0)$, $v(\xi(x,0))$ is bounded away from $0$ (and it is always bounded above by $1$).
Therefore, $\tau(\cdot)$ is Lipschitz and, moreover, has the derivative bounded away from $0$; then so is 
its inverse function which we denote by $y(\tau), ~\tau \ge 0$. Finally, we formally define
$$
\xi(x,t) = \xi(x - y(t),0), ~~t\ge 0.
$$
To summarize, for any initial deterministic stage-profile $\phi(x,0), ~0\le x \le 1,$ with bounded density $\xi(x,0)$,
we formally (and rigorously) defined the unique stage-profile $\phi(x,t), ~0\le x \le 1,$ for any $t\ge 0$,
where $\phi(x,t) = \int_0^x \xi(z,t) dz$. The object $\phi(\cdot,\cdot)$ we will call the {\em hydrodynamic model} (of the reshuffling system) with initial state $\phi(\cdot,0)$. In the process of formally defining a hydrodynamic model, we have rigorously proved 
the following 
\begin{theorem}
\label{th-hydro}
Let $\lambda<1$ be fixed. 
Consider any hydrodynamic model $\phi(\cdot,\cdot)$ with the initial state $\phi(x,0), ~0\le x \le 1,$ having
 bounded density $\xi(x,0)$. Then, as $t\to\infty$, its density $\xi(\cdot,t)$ uniformly converges to the 
 constant density $\psi$.
\end{theorem}

\begin{rem}
It is not difficult to rigorously define a hydrodynamic model for an arbitrary initial state $\phi(\cdot,0)$ with finite total mass, $\phi(1,0)<\infty.$
So, $\phi(x,0)$ does not have to be absolutely continuous or even continuous.
We do not do this to this, because it adds technical details, without giving new intuition.
\end{rem}

Our conjecture about the dynamics of the stage-profile in the reshuffling system is as follows.
\begin{conj}
\label{conj-rs-dynamics}
Let $\lambda<1$ be fixed. 
If the initial stage-profile $\phi_{RS}^N(\cdot,0)$ in the reshuffling system converges to a deterministic 
stage-profile $\phi(\cdot,0)$ (with bounded density), then the process $\phi_{RS}^N(\cdot,\cdot)$ converges
(in appropriate sense) to the hydrodynamic model $\phi(\cdot,\cdot)$ with initial state $\phi(\cdot,0)$.
\end{conj}

This conjecture complements the steady-state Conjectures~\ref{conj:reshuf-Poisson}(iv) and \ref{conj-reshuf-location}.
In particular, Conjecture~\ref{conj-rs-dynamics} in conjunction with Theorem~\ref{th-hydro} implies
that as time $t\to\infty$ goes to infinity, the normalized sojourn time of a packet entering the system at time $t$
converges to $-\log(1-\lambda)/\lambda$.

\subsection{Example of the dynamics of the reshuffling system. Comparison to that under RU}

In this subsection we run a simulation experiment to test Conjecture~\ref{conj-rs-dynamics} about the dynamics of the reshuffling system. Specifically, we consider an extreme scenario when the entire particle mass (``impulse'')  is initially located at point $0$ and there are no new arrivals, i.e. $\lambda=0$. 

Consider a hydrodynamic model with initial state $\phi(\cdot,0)$, with density
\begin{align*}
	\xi(x, 0) = 
	\begin{cases}
 		u, & 0 \leq x \leq a \\
 		0, & a < x \le 1
 	\end{cases}
\end{align*}
where $a \in (0, 1)$. Then, $\phi(\cdot,0)$ is such that: the $a$-long interval, where the density is $u$, will move right at the constant speed $1$ until time $1/2-a$, when it ``hits'' the middle point $1/2$; then for the time $a/[(1-e^{-u})/u]=au/(1-e^{-u})$, while it overlaps with the middle point $1/2$, it will move at the constant speed $(1-e^{-u})/u$; and finally, for the time $1/2$,
it again will move at speed $1$ until the entire particle mass leaves the system. Now, if we keep the total initial mass $c=au$
constant, and let $u\to\infty$, in the limit we obtain the following hydrodynamic model: it has the particle mass 
(``impulse'') $c$ initially concentrated at $0$; this mass first moves at speed $1$ until it reaches the middle point $1/2$; then the mass stays at $1/2$ for the time $c$; then it resumes the movement at speed $1$ until reaching end point $1$. 
In particular, sojourn time of the particle mass in the transformed system is $c+1$, which is the approximation of the normalized delay of the packets when $N$ is large.

Figure~\ref{fig:impulse-scale-5} 
show simulation results for a system with finite $N=10000$,
initialized with $c A_N$ stage $1$ packets, and with no new exogenous packet arrivals. Therefore, $c$ is the ``impulse size'' of the corresponding hydrodynamic model. We consider $c=5$ and $c=10$. The figure shows the average time for a packet to reach point $x$ in the transformed system. We see that each plot does resemble that for 
the corresponding hydrodynamic model, albeit the curve is smooth since we simulate a finite system; a packet first moves fast (at speed close to $1$), then slows down in the middle of the transformed interval, then accelerates again (to speed close to $1$).
The slowdown is close to the limiting slowdown of $c+1$.

Figure~\ref{fig:impulse-scale-5}  
also shows plots for the corresponding systems under RU discipline. In this case the system is initialized so that
 $c A_N$ stage $1$ packets are placed independently at nodes chosen uniformly at random.
We compare RS and RU disciplines, because, in a sense, RS is a ``simplified version'' of RU, 
both choosing a random useful packet to transmit on a link. We observe the the delay profile under RU is quite similar
to that of RS. In fact, the plots practically coincide up to some point $x^*$ {\em which is greater than $1/2$.} 
Recall, that majority of stages $i$ in the original system are located close to point $1/2$
in the transformed system. We conclude that, for the RU and RS systems with impulse initial state, the delays 
to reach the majority of stages $i$, except a small fraction of stages ``at the end,'' are practically same.
For the small fraction of stages at the end, however, the delays under RU are larger. 
Informally speaking, this is due the fact that the conjectures that we made about the behavior of an RS system
(which are very well confirmed by simulations), do {\em not} appear to hold under RU.

\begin{figure}[htp!]
	\centering
		\begin{subfigure}[t]{0.45\textwidth}
        \centering
        \includegraphics[scale=0.50]{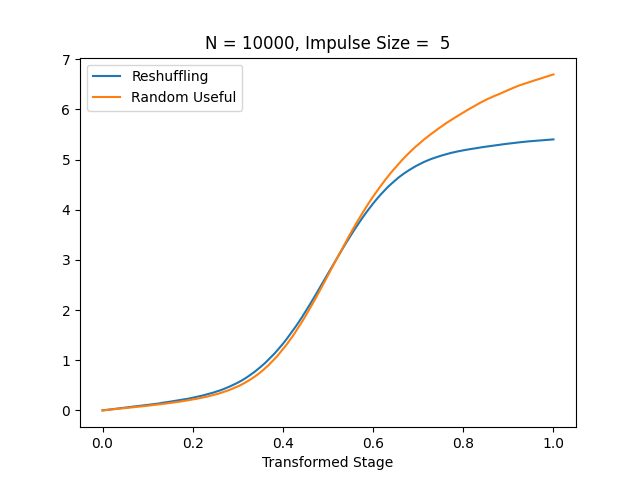}
    \end{subfigure}%
    ~ 
    \begin{subfigure}[t]{0.48\textwidth}
        \centering
        \includegraphics[scale=0.50]{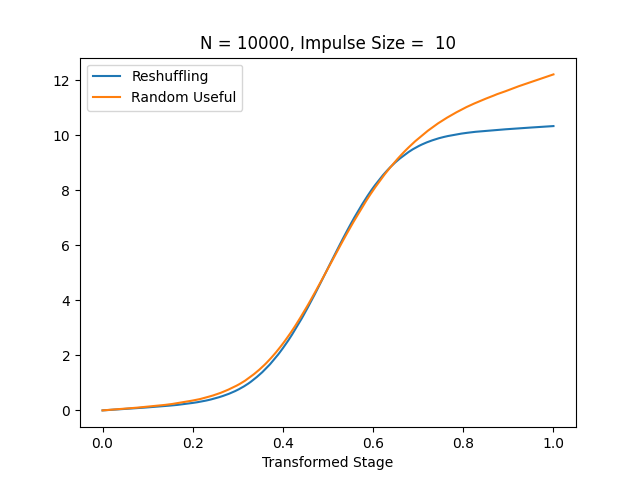}
    \end{subfigure}
    \caption{Mean time in transformed system to reach point $x$. Impulse sizes $5$ and $10$.}
    \label{fig:impulse-scale-5}	
\end{figure}

\newpage

\section{Additional Simulation Data for Sections~\ref{sec:comparisons}, \ref{sec:delays} and \ref{sec:reshuffling}}
\label{sec-supporting-simulation}

\begin{table}
\centering
\begin{tabular}[]{| c ||c | c | c |}
	\hline
	$N = 50$, $\lambda = 0.3\ $ & Distinct Packets & Undelivered Packet Copies & Age of Information \\
	\hline
	OU & $3.538$ & $96.350$ & $0.175$ \\
	RU & $3.604$ & $84.294$ & $0.179$ \\
	Selfish & $3.394$ & $85.462$ & $0.163$ \\
	\hline
\end{tabular}
\begin{tabular}[]{| c ||c | c | c |}
	\hline
	$N = 100$, $\lambda = 0.3$ & Distinct Packets & Undelivered Packet Copies & Age of Information \\
	\hline
	OU & $4.078$ & $225.127$ & $0.105$ \\
	RU & $4.185$ & $193.384$ & $0.107$ \\
	Selfish & $3.861$ & $198.076$ & $0.097$ \\
	\hline
\end{tabular}
\begin{tabular}[]{| c ||c | c | c |}
	\hline
	$N = 150$, $\lambda = 0.3$ & Distinct Packets & Undelivered Packet Copies & Age of Information \\
	\hline
	OU & $4.402$ & $364.837$ & $0.077$ \\
	RU & $4.517$ & $310.902$ & $0.079$ \\
	Selfish & $4.170$ & $324.388$ & $0.071$ \\
	\hline
\end{tabular}
\begin{tabular}[]{| c ||c | c | c |}
	\hline
	$N = 200$, $\lambda = 0.3$ & Distinct Packets & Undelivered Packet Copies & Age of Information \\
	\hline
	OU & $4.628$ & $512.799$ & $0.061$ \\
	RU & $4.791$ & $437.627$ & $0.063$ \\
	Selfish & $4.366$ & $455.245$ & $0.057$ \\
	\hline
\end{tabular}
\begin{tabular}[]{| c ||c | c | c |}
	\hline
	$N = 50$, $\lambda = 0.5\ $ & Distinct Packets & Undelivered Packet Copies & Age of Information \\
	\hline
	OU & $7.897$ & $231.623$ & $0.282$ \\
	RU & $7.956$ & $174.702$ & $0.307$ \\
	Selfish & $7.131$ & $177.919$ & $0.248$ \\
	\hline
\end{tabular}
\begin{tabular}[]{| c ||c | c | c |}
	\hline
	$N = 100$, $\lambda = 0.5$ & Distinct Packets & Undelivered Packet Copies & Age of Information \\
	\hline
	OU & $9.146$ & $543.397$ & $0.165$ \\
	RU & $9.261$ & $394.447$ & $0.180$ \\
	Selfish & $8.177$ & $422.581$ & $0.144$ \\
	\hline
\end{tabular}
\begin{tabular}[]{| c ||c | c | c |}
	\hline
	$N = 150$, $\lambda = 0.5$ & Distinct Packets & Undelivered Packet Copies & Age of Information \\
	\hline
	OU & $9.867$ & $884.421$ & $0.120$ \\
	RU & $10.096$ & $641.748$ & $0.131$ \\
	Selfish & $8.801$ & $696.006$ & $0.104$ \\
	\hline
\end{tabular}
\begin{tabular}[]{| c ||c | c | c |}
	\hline
	$N = 200$, $\lambda = 0.5$ & Distinct Packets & Undelivered Packet Copies & Age of Information \\
	\hline
	OU & $10.393$ & $1245.166$ & $0.095$ \\
	RU & $10.693$ & $897.505$ & $0.104$ \\
	Selfish & $9.235$ & $986.223$ & $0.082$ \\
	\hline
\end{tabular}
\begin{tabular}[]{| c ||c | c | c |}
	\hline
	$N = 50$, $\lambda = 0.7\ $ & Distinct Packets & Undelivered Packet Copies & Age of Information \\
	\hline
	OU & $18.318$ & $583.120$ & $0.501$ \\
	RU & $17.242$ & $345.818$ & $0.599$ \\
	Selfish & $14.921$ & $353.674$ & $0.406$ \\
	\hline
\end{tabular}
\begin{tabular}[]{| c ||c | c | c |}
	\hline
	$N = 100$, $\lambda = 0.7$ & Distinct Packets & Undelivered Packet Copies & Age of Information \\
	\hline
	OU & $21.602$ & $1407.346$ & $0.297$ \\
	RU & $20.235$ & $780.674$ & $0.348$ \\
	Selfish & $16.816$ & $843.322$ & $0.228$ \\
	\hline
\end{tabular}
\begin{tabular}[]{| c ||c | c | c |}
	\hline
	$N = 150$, $\lambda = 0.7$ & Distinct Packets & Undelivered Packet Copies & Age of Information \\
	\hline
	OU & $23.691$ & $2331.977$ & $0.217$ \\
	RU & $22.039$ & $1249.37$ & $0.251$ \\
	Selfish & $18.123$ & $1406.262$ & $0.164$ \\
	\hline
\end{tabular}
\begin{tabular}[]{| c ||c | c | c |}
	\hline
	$N = 200$, $\lambda = 0.7$ & Distinct Packets & Undelivered Packet Copies & Age of Information \\
	\hline
	OU & $24.987$ & $3337.345$ & $0.172$ \\
	RU & $23.470$ & $1745.280$ & $0.200$ \\
	Selfish & $19.249$ & $2029.730$ & $0.130$ \\
	\hline
\end{tabular}

\caption{Simulation results for average values of performance metrics, under different disciplines, in a symmetric system.}
	\label{tab:sim-data} 
\end{table}

\begin{table}
\centering
\begin{tabular}[]{| c ||c | c |}
	\hline
	OU, $\lambda = 0.3$ & Mean Sojourn Time & Variance \\
	\hline
	$N = 50$ & $1.315$ & $0.100$ \\
	$N = 100$ & $1.315$ & $0.079$ \\
	$N = 150$ & $1.314$ & $0.071$ \\
	$N = 200$ & $1.312$ & $0.066$ \\
	\hline
\end{tabular}

\begin{tabular}[]{| c ||c | c |}
	\hline
	OU, $\lambda = 0.5$ & Mean Sojourn Time & Variance \\
	\hline
	$N = 50$ & $1.764$ & $0.210$ \\
	$N = 100$ & $1.767$ & $0.176$ \\
	$N = 150$ & $1.767$ & $0.160$ \\
	$N = 200$ & $1.768$ & $0.151$ \\
	\hline
\end{tabular}

\begin{tabular}[]{| c ||c | c |}
	\hline
	OU, $\lambda = 0.7$ & Mean Sojourn Time & Variance \\
	\hline
	$N = 50$ & $2.908$ & $0.666$ \\
	$N = 100$ & $2.989$ & $0.608$ \\
	$N = 150$ & $3.026$ & $0.583$ \\
	$N = 200$ & $3.060$ & $0.573$ \\
	\hline
\end{tabular}

\begin{tabular}[]{| c ||c | c |}
	\hline
	RU, $\lambda = 0.3$ & Mean Sojourn Time & Variance \\
	\hline
	$N = 50$ & $1.339$ & $0.145$ \\
	$N = 100$ & $1.347$ & $0.118$ \\
	$N = 150$ & $1.353$ & $0.106$ \\
	$N = 200$ & $1.356$ & $0.100$ \\
	\hline
\end{tabular}

\begin{tabular}[]{| c ||c | c |}
	\hline
	RU, $\lambda = 0.5$ & Mean Sojourn Time & Variance \\
	\hline
	$N = 50$ & $1.774$ & $0.375$ \\
	$N = 100$ & $1.793$ & $0.317$ \\
	$N = 150$ & $1.809$ & $0.298$ \\
	$N = 200$ & $1.820$ & $0.281$ \\
	\hline
\end{tabular}

\begin{tabular}[]{| c ||c | c |}
	\hline
	RU, $\lambda = 0.7$ & Mean Sojourn Time & Variance \\
	\hline
	$N = 50$ & $2.753$ & $1.192$ \\
	$N = 100$ & $2.794$ & $1.065$ \\
	$N = 150$ & $2.839$ & $1.024$ \\
	$N = 200$ & $2.859$ & $0.992$ \\
	\hline
\end{tabular}
\caption{Simulation results for mean and variance of normalized sojourn time, for OU and RU disciplines, in a symmetric system.}
	\label{tab:slowdown}
\end{table}

 \begin{figure}[htp!]
	\centering
	\includegraphics[scale=0.5]{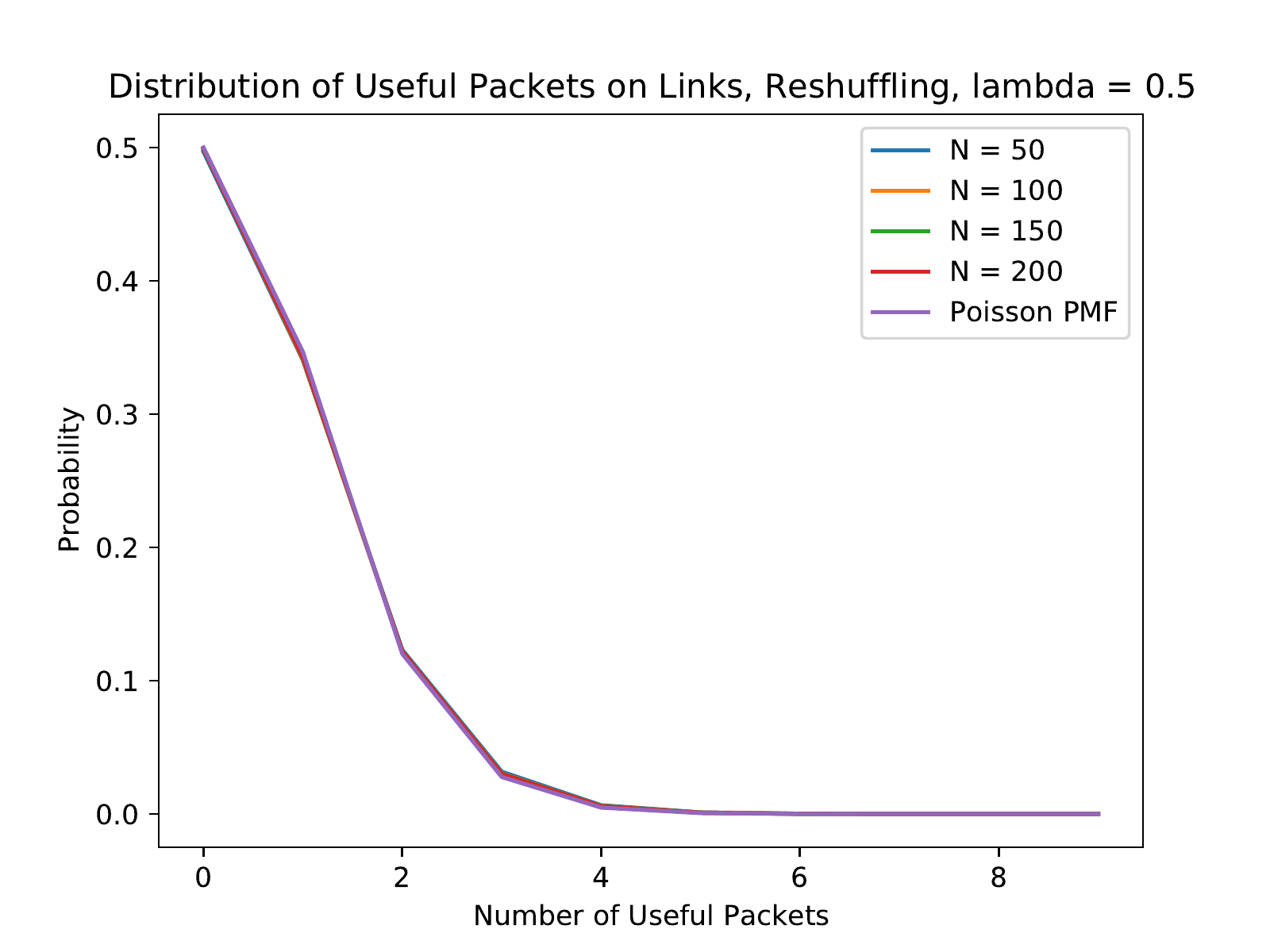}
	\caption{The simulated distribution of the number of useful packets on a fixed link of reshuffling system, with $\lambda = 0.5$.
	The distribution appears to converge to Poisson.}
	\label{fig:useful-packets-distn}
\end{figure}

\end{document}